\documentclass[12pt,a4paper]{article}
\usepackage{DPUmath}
\usepackage[margin=1.25in]{geometry}

\author{Erik Dahlen}
\date{}

\begin{document}
\begin{titlepage}

   \begin{center}
       \vspace*{1cm}

        \large
       \textbf{On the $b$-chromatic number of star graph operators}

    \normalsize
        \textit{Combinatorics and Graph Theory}
     
        \textit{Erik Dahlen$^1$}

\vspace{0.8cm}

\begin{abstract}
A $b$-coloring is a proper coloring such that for each color class, there exists at least one vertex that is adjacent to at least one vertex in every other color class.
The $b$-chromatic number of a graph $G$ is the maximum number $k$ such that $G$ admits a $b$-coloring with $k$ colors.
This paper focuses on the $b$-chromatic number of the power graph of the Cartesian product of star graphs.  
In addition, we also study the total graph and the line graph of the Cartesian product of star graphs. 
Our main result generalizes the result shown in \cite{qn} on the b-chromatic number of the Cartesian product of two stars. We find exact values for the b-chromatic number of particular Cartesian products of complete graphs and explore the bounds of the generalized Cartesian product of complete graphs.
\end{abstract}
      \end{center}
      
       \textbf{Key words}: $b$-coloring, $b$-chromatic number, Cartesian product, star graph, graph power, graph operators, complete graph.

        \textbf{AMS subject classification (2020)}: 05C69 

\vspace{2.0cm}

\end{titlepage}

\section{Introduction}

\addtocounter{page}{1}

In this paper, we study the $b$-chromatic number of numerous graph families that stem from the star graph. All graphs in this paper are finite, simple, undirected, and connected. Denote this graph $G(V,E)$ with the vertex set $V(G)$ and the edge set $E(G)$.

For any graph $G$, a \textit{proper coloring} assigns a color to every vertex $v\in V(G)$ such that no adjacent vertices share the same color. All colors exist in a color set $C$. It is common to represent colors as integers such that $C = \{1,2,3,4,... k\}$ for some $k\in\mathbb{Z}$. We can partition the set of vertices of $G$ into disjoint color classes, $C_1,\ldots, C_k$, such that $C_i$ is the set of vertices with color $i$.

\begin{ex} Below in Figure~\ref{figure1}, is an example of a proper coloring on a path of 4 vertices, $P_4$.
\begin{figure}[h]
\centering
 \begin{tikzpicture}
	\begin{pgfonlayer}{nodelayer}
		\node [style=red] (0) at (-4, 0) {};
		\node [style=blue] (1) at (-3, 0) {};
		\node [style=red] (2) at (-2, 0) {};
		\node [style=blue] (3) at (-1, 0) {};
	\end{pgfonlayer}
	\begin{pgfonlayer}{edgelayer}
		\draw (0) to (1);
		\draw (1) to (2);
		\draw (2) to (3);
	\end{pgfonlayer}
    \end{tikzpicture}
    \caption{Proper coloring of $P_4$.}
    \label{figure1}
\end{figure}
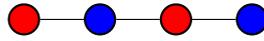

\end{ex}

Something of immediate interest is the smallest number of colors for any proper coloring of a graph $G$. The number of colors satisfying these conditions is called the \textit{chromatic number} of $G$, denoted $\chi(G)$. On the other hand, the $b$-chromatic number, $\varphi(G)$, introduced by Irving and Manlove in 1997\cite{Manlove}, can intuitively be used to denote the \textit{maximum} number of colors used to color a graph such that each color class is adjacent to all other colors. We define the $b$-chromatic number precisely below,

\begin{defn}\label{Definition 2} 
In a proper coloring of $G$ with $k$ colors, a vertex $v$ with color $i$ is called a \textit{$b$-vertex} if for each color $j\neq i$, $v$ is adjacent to at least one vertex of color $j$. A proper coloring is a \textit{$b$-coloring} if there exists a $b$-vertex for each color $i\in[k]$.
The \textit{b-chromatic number}, denoted $\varphi(G)$, is the largest integer $k$ such that there exists a $b$-coloring with $k$ colors.

\end{defn}

\begin{ex}

Below is Figure~\ref{figure2}, which shows the \textit{b-coloring} of the graph $S_3\square S_3$, using 5 colors. In Theorem \ref{Theorem_box_prod}, we will show that this must be its maximal $b$-coloring. 

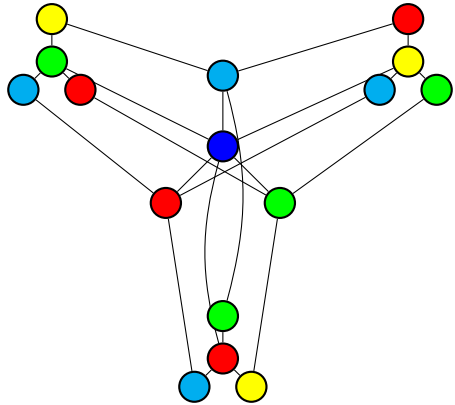
\begin{figure}[h]
    \centering
     \begin{tikzpicture}[scale=0.75]
	\begin{pgfonlayer}{nodelayer}
		\node [style=blue] (1) at (0, 0) {};
		\node [style=aqua] (4) at (0, 1.25) {};
		\node [style=gn] (5) at (1, -1) {};
		\node [style=gn] (6) at (-3, 1.5) {};
		\node [style=yn] (7) at (-3, 2.25) {};
		\node [style=aqua] (8) at (-3.5, 1) {};
		\node [style=red] (9) at (-2.5, 1) {};
		\node [style=red] (10) at (3.25, 2.25) {};
		\node [style=aqua] (11) at (-0.5, -4.25) {};
		\node [style=aqua] (12) at (2.75, 1) {};
		\node [style=gn] (13) at (3.75, 1) {};
		\node [style=yn] (14) at (3.25, 1.5) {};
		\node [style=red] (15) at (-1, -1) {};
		\node [style=red] (16) at (0, -3.75) {};
		\node [style=gn] (17) at (0, -3) {};
		\node [style=yn] (18) at (0.5, -4.25) {};
	\end{pgfonlayer}
	\begin{pgfonlayer}{edgelayer}
		\draw (6) to (9);
		\draw (6) to (7);
		\draw (6) to (8);
		\draw (14) to (12);
		\draw (14) to (13);
		\draw (14) to (10);
		\draw (17) to (16);
		\draw (16) to (18);
		\draw (16) to (11);
		\draw (1) to (15);
		\draw (1) to (5);
		\draw (1) to (4);
		\draw (6) to (1);
		\draw (14) to (1);
		\draw [bend left=15] (16) to (1);
		\draw (7) to (4);
		\draw (9) to (5);
		\draw (8) to (15);
		\draw (10) to (4);
		\draw (12) to (15);
		\draw (13) to (5);
		\draw [bend right=15] (17) to (4);
		\draw (18) to (5);
		\draw (11) to (15);
	\end{pgfonlayer}
\end{tikzpicture}
\caption{$b$-coloring of $G$ such that $\varphi(G)=5$}
\label{figure2}
\end{figure}

\end{ex}

A common technique for finding the $b$-chromatic number of a particular graph $G$ is to calculate bounds on $\varphi(G)$. Well-known lower bounds of the $b$-chromatic number are the largest complete sub-graph in $G$, known as the clique number $\omega(G)$, and the chromatic number. The chromatic number is always greater than or equal to the clique number; hence, the chromatic number is generally used as the lower bound of $\varphi(G)$.
The general upper-bound used throughout this paper and in other works is the $m$-degree of $G$, defined as the size of the largest subset $S\subseteq V(G)$ such that each vertex has a degree of at least $|S|-1$. The $m$-degree can be used to find the \textit{feasible} number of colors that could be used to form a $b$-coloring.

In this paper, we study the $b$-chromatic number of star graphs, products of stars, and the power graph of star graphs (see Definitions \ref{boxproddef} and \ref{powergraphdef}). We write $S_n$ for the star graph with a central vertex $v_0$, which is adjacent to $n$ other vertices such that none of the $n$ vertices are adjacent to any of the other $n-1$ vertices. To avoid confusion, we will always refer to $v_0$ as the central vertex and $V(S_n)\setminus \{v_0\}$ as non-central vertices. The following is our first main result.

\begin{thm}[Theorem 4.2]
 \textit{Let $G$ be the $k$-th graph power of the Cartesian product of two graphs $S_n$ and $S_m$, denoted $(S_n\square S_m)^k$, where $n\ge m$. Then the following are true of $\varphi(G)$,}

    \begin{enumerate}
        \item \textit{If $k=1$, then $\varphi(G)= m+2$}

        \item \textit{If $k=2$, then $\varphi(G) = \begin{cases}
            m+n+1, & \textit{If $n>m$}\\
            2n+2, & \textit{If $n=m$}
        \end{cases}$}

        \item \textit{If $k=3$ and $n\ge m$, then;}
        \begin{enumerate}
            \item If $m\leq n< m(m-1)$, then $2n+m+1\leq \varphi(G)\leq m^2+n+1$.
            \item If $n\ge m(m-1)$, then $\varphi(G)=2n+m+1$.
        \end{enumerate}

        \item \textit{If $k\ge4$, then $\varphi(G)= nm+n+m+1$}
    \end{enumerate}
    
\end{thm}

This theorem is motivated by the study of the b-chromatic number of non-regular graph powers.
The $b$-chromatic number has been computed for trees, regular graphs, cubic graphs, and Cartesian products; under specific conditions \cite{Manlove, Chuan, Klavzar, Jakovac, Mekkia, Maffray, qn}. In particular, power graphs provide specific recursive relation structures seen in computer networks and are thus well studied in theoretical computer science \cite{Brice}. Furthermore, Cartesian products provide us with generalizations to higher dimensions. 
For information on simple power graphs and regular power graphs, consider \cite{Brice} and Definition \ref{powergraphdef}. In addition to stars, this paper is interested in the \textit{rook graph} or the Cartesian product of two complete graphs. Bounds for the $b$-chromatic number of the rook graph are known, as shown in \cite{qn, Omoomi}, but in Theorem \ref{thm4.2} we give explicit computations under certain conditions and tighter bounds for certain cases of the rook graph.

Our next main result is a generalization of \cite[Proposition~7, part (a)]{qn}.

\begin{thm}[Theorem 3.1]
     \textit{For two graphs $S_n$ and $S_m$, where $n\ge m$, $\varphi(S_n\square S_m)=m+2$.}
\end{thm}

Using Theorem \ref{Theorem_box_prod}, we arrive at the $b$-chromatic number of the line graph and total graph of $S_n\square S_m$.

\begin{thm}[Theorem 3.2]
\textit{For two graphs $S_n$ and $S_m$,} $\varphi(L(S_n\square S_m))=m+n$
\end{thm}

\begin{thm}[Theorem 3.3]
\label{thm:line_tot}
\textit{For $n\ge m\ge 3$, It follows that,} $$\varphi(T(S_n\square S_m))=\begin{cases}
    2m+n+1, & \textit{If $n> 2(m-1)$}\\
    2n+3 & \textit{If $n\leq 2(m-1)$}
\end{cases}$$
\end{thm}

In Section~\ref{sec:prelims}, we will cover the necessary definitions of the $b$-chromatic number and the propositions that come directly from the definitions. In Section~\ref{sec:products}, we will cover the main theorems of this paper, exploring the Cartesian Product of two stars, the total graph of the Cartesian product of two stars, and the line graph of the Cartesian product of two stars. In Section~\ref{sec:powers}, we explore some results on the graph power of the Cartesian product of two stars and other corollaries.

\section{Preliminaries}\label{sec:prelims}

In this section, we will describe the definitions and propositions used throughout the paper.


A vertex $v\sim u$ if and only if, $\{u,v\}\in E(G)$. Furthermore, the set of neighbors of a given vertex $v\in G$ can be represented as $N(v) = \{u\in V(G): u\sim v\}$. The \textit{degree} of any $v\in G$ is therefore $|N(v)|$. Given a proper coloring of the graph $G$, we may often refer to a \textit{color} $c$ being in the neighborhood of a vertex, $v$, or a color $c$ being adjacent to $v$, as a shorthand of saying that a vertex of the color $c$ is adjacent to $v$.

We can observe that if $\varphi(G)=n$, then there must exist a sequence of vertices $v_1,v_2,\ldots,v_n$, or the $b$-vertices, such that each one of these vertices has a degree of \textit{at least} $n-1$.

\begin{defn} 

    Let $G$ be a graph with the vertices $v_1,v_2,\ldots, v_n$, such that $deg(v_1)\ge deg(v_2)\ge\dots\ge deg(v_n)$.
    The \textit{m-degree} of $G$, denoted $m(G)$ is defined by, \begin{equation*}
        m(G) := max\{i: deg(x_i)\ge i-1\}
    \end{equation*}
\end{defn}

The following useful proposition is immediate.

\begin{prop}\label{squeeze}
    $\omega(G)\leq\chi(G)\le\varphi(G)\le m(G)$ \cite{Manlove}
\end{prop}

By this definition, $\varphi(G)\leq m(G)$. For many graphs, the $b$-chromatic number will either be equal to $m(G)$ or $m(G)-1$.

\begin{prop}
    \textit{Let $G=S_n$. Then,}
    \begin{equation*}
        \varphi(G)=\begin{cases}
            1 & \textit{, if $n=0$}\\
            2 & \textit{, if $n>0$}
        \end{cases}
    \end{equation*}
\end{prop}

\begin{proof}
    $S_0$ is the trivial graph, $K_1$. As there exists only one vertex to color, then 1 is the maximum number of colors. Let $n>0$. Assigning all of the vertices $v_1,\ldots,v_n$ the color 2 will form a proper $b$-coloring. The chromatic number is 2 and $m(S_n)=2$. We conclude that for $n>0$, $\varphi(S_n)=2$.
\end{proof}

\begin{defn}\label{boxproddef}

Let the vertices $V(G) = \{u_0,u_1,\ldots u_n\}$ and the vertices $V(H) =  \{v_0,v_1,\ldots v_m\}$.
The \textit{Cartesian product} of two graphs $G$ and $H$ defined such that $V(G\square H) = V(G) \times V(H)$ and vertices $(u,v),(u',v')\in V(G\square H)$ are adjacent if, and only if, either $v=v'$ and $u\sim u'$, \textit{or} $u=u'$ and $v\sim v'$. 
    We partition $G\square H$ into $m+1$ disjoint sub-graphs, each of which is isomorphic to $G$ as follows,
     $$\bigsqcup_{v\in V(H)} (G,v). $$
    We call any sub-graph $(G,v)$ an \emph{inner graph} of $G\square H$ and we write it as  $(G)_v$. 
    We also call the graph $H$ the \emph{skeleton}. 
   We will very often use the similar notation $(u)_{v}$, equivalent to saying $(u,v)\in V(G)\times V(H)$ to denote the vertex $u$ that is contained within the inner graph $(G)_v$. This usage can also be extended to any \textit{edge}, denoted, $\left(\left\{u,u'\right\}\right)_{v}$.

   Observe that $G\square H\cong H\square G$.

\end{defn}

\begin{ex}
    Below is the box product of two path graphs, $P_3\square P_3$, often referred to as the \textit{Lattice} Graph.

    \begin{figure}[h]
        \centering
        \begin{tikzpicture}
	\begin{pgfonlayer}{nodelayer}
		\node [style=White] (0) at (-4, 0) {};
		\node [style=White] (1) at (-3, 0) {};
		\node [style=White] (2) at (-2, 0) {};
            \node [style=White] (3) at (-4, -1) {};
		\node [style=White] (4) at (-3, -1) {};
		\node [style=White] (5) at (-2, -1) {};
            \node [style=White] (6) at (-4, -2) {};
		\node [style=White] (7) at (-3, -2) {};
		\node [style=White] (8) at (-2, -2) {};
	\end{pgfonlayer}
	\begin{pgfonlayer}{edgelayer}
		\draw (0) to (1);
		\draw (1) to (2);
            \draw (3) to (4);
		\draw (4) to (5);
            \draw (6) to (7);
		\draw (7) to (8);
        \draw (0) to (3);
        \draw (3) to (6);
        \draw (1) to (4);
        \draw (4) to (7);
        \draw (2) to (5);
        \draw (5) to (8);
	\end{pgfonlayer}
    \end{tikzpicture}
        \caption{The Cartesian Product of $P_3$ and $P_3$.}
        
    \end{figure}
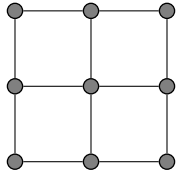
\end{ex}

The Cartesian product is one of many types of products that have been well-studied. Other product graphs that have been studied include the tensor product, the direct product, the rooted product, the Kronecker product, and others. Consider \cite{Chuan, Jakovac, Koch}.

\begin{prop}\label{prop2.2}
$\varphi(G\square H) \ge max\{\varphi(G),\varphi(H)\}$
\end{prop}

\begin{proof}
    Shown in \cite{qn}.
\end{proof}

This proposition proves useful when we consider better lower bounds than the chromatic number. More often than not, $\chi(G\square H)=\max\{\varphi(G),\varphi(H)\}$.

\begin{defn}
    The \textit{line graph}, denoted as $L(G)$, is a graph with a vertex set $V$ such that each $v\in V$ corresponds to an element $e\in E(G)$. For any $v,u\in L(G)$, $v\sim u$ iff edges are adjacent to other vertices in $G$.

    When we color the vertices of $L(G)$, we instead color the edges of $G$. A \textit{ proper edge coloring} assigns a color $c(e)$ to an edge $e\in E(G)$ such that no two edges that share the same vertex have the same color.
\end{defn}

\begin{prop}
    Each proper coloring of the line graph corresponds to a proper edge coloring. Each proper coloring of the total graph corresponds to a proper edge-vertex coloring.
\end{prop}

\begin{defn}
    The \textit{total graph}, denoted as $T(G)$, is a graph with a vertex set $V$ such that each $v\in V$ corresponds to an element of $V(G)\cup E(G)$. For any $v,u\in T(G)$, $v\sim u$ iff vertices in $G$ are adjacent to other vertices, vertices are incident to other edges in $G$, or edges are adjacent to other edges.

    Similar to the line graph, we have a \textit{proper edge-vertex coloring} that has the further criteria that if an edge $e\in E(G)$ contains the vertex $v\in V(G)$, then $c(e)\neq c(v)$.
\end{defn}

The total graph is a generalization of the line graph.

\begin{prop}
    \textit{For $n>1$, the $b$-chromatic number of $L(S_n)=n-1$.}
\end{prop}

\begin{proof}
    The line graph of $S_n$ is isomorphic to the complete graph, $K_{n-1}$. Thus, the $b$-chromatic number of $L(S_n)$ must be $n-1$.
\end{proof}

\begin{prop}
    \textit{For $n>1$, the $b$-chromatic number of $T(S_n)=n+1$.}
\end{prop}

\begin{proof}
    The largest clique in $T(S_n)$ is always $n+1$. This is formed by the $n$ edges and the central vertex $v_0\in S_n$. Additionally, the $m$-degree is also always $n+1$. By Proposition~\ref{squeeze}, the $b$-chromatic number of $T(S_n)$ must be $n+1$.
\end{proof}

\begin{defn}
    $d(u,v)$ represents the distance between the vertices $u$ and $v$. The \textit{distance} between $u$ and $v$ is the shortest path of edges from $u$ to $v$. The \textit{Diameter} of a graph represents the longest distance between two points $u,v\in G$.
\end{defn}

\begin{defn}\label{powergraphdef}\label{power}
    The power graph of $G$, denoted as $G^p$, has a vertex set $V(G)$ and an edge set $E(G^p) = \{\{u,v\}\in E(G) : d(u,v)\le p\}$.
\end{defn}

    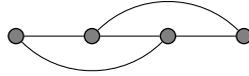
\begin{figure}[h]
        \centering
        \begin{tikzpicture}
	\begin{pgfonlayer}{nodelayer}
		\node [style=White] (0) at (-4, 0) {};
		\node [style=White] (1) at (-3, 0) {};
		\node [style=White] (2) at (-2, 0) {};
		\node [style=White] (3) at (-1, 0) {};
	\end{pgfonlayer}
	\begin{pgfonlayer}{edgelayer}
		\draw (0) to (1);
		\draw (1) to (2);
		\draw (2) to (3);
        \draw [bend right=45](0) to (2);
        \draw [bend left=45](1) to (3);
	\end{pgfonlayer}
    \end{tikzpicture}
        \caption{The second power of $P_4$.}
    \end{figure}

\section{Cartesian Product of Star Graphs}\label{sec:products}

In this section, we prove the explicit formulas for the $b$-chromatic number of $S_n\square S_m$, $L(S_n\square S_m)$, and $T(S_n\square S_m)$.

\begin{thm}\label{Theorem_box_prod}
   \textit{For two graphs $S_n$ and $S_m$, where $n\ge m\ge1$, $\varphi(S_n\square S_m)=m+2$.}
\end{thm}

Before we dive into the proof of Theorem~\ref{Theorem_box_prod}, we prove the following useful lemma.

In Lemma 1 and Theorem 3.1, we use the following notation. 
Without loss of generality, let $n\ge m$. We write $V(S_m)=\{v_0, v_1, \ldots, v_m\}$ where $v_0$ is the central vertex of $S_m$ and $V(S_n) = \{w_0, w_1, \ldots, w_n\}$ where $w_0$ is the central vertex of $S_n$. Observe that there are $(m+1)$ central vertices for each inner graph $(S_n)_{v_j}$ by the definition of an inner graph.

\begin{lemma}\label{lem_box_prod}
\textit{For $1<m\leq n$, the $m$-degree of $S_n\square S_m$ is $m + 2$.}
\end{lemma}
\begin{proof}
Without loss of generality, assume that $n\ge m$. We can decompose $S_n\square S_m$ such that $S_m$ is the \textit{skeleton}. Then, there must be $m+1$ inner graphs of the form $(S_n)_{v_j}$ for $j\in\{0,\ldots,m\}$. Each vertex of the form $(w_0)_{v_j}$ such that $j\neq0$ has a degree $n+1$. If $n\ge m$, then $n+1\ge m+1$.
Furthermore, $(w_0)_{v_0}$ has a degree of $n+m$ and $n+m\ge n+1\ge m+1$ and therefore, $n+m\ge m+1$.
Finally, the vertices $(w_i)_{v_0}$ such that $i\neq0$  have a degree of $m+1$ and by reflexivity, $m+1\ge m+1$.  
Therefore, there are \textit{at least} $m+2$ vertices of degree $m+1$, satisfying the conditions for $m(G) = m+2 = m + 2$.
\end{proof}

\begin{proof}[Proof of Theorem~\ref{Theorem_box_prod}]\

The method used to color the graph $S_n\square S_m$ as a $b$-coloring can be described algorithmically as follows,

\begin{enumerate}

    \item Color each $(w_0)_{v_j}$ with color $i\in \{0,1,\ldots, m\}$.

    \item  Without loss of generality, assign the vertex $(w_1)_{v_0}$ the color $(m+1)$. These specific vertices will be $b$-vertices. Observe that $(w_0)_{v_0}$ is already a $b$-vertex.

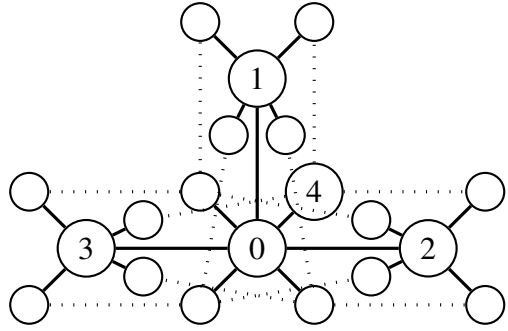
\begin{center}
\begin{tikzpicture}[
every edge/.style = {draw=black,very thick},
 vrtx/.style args = {#1/#2}{%
      circle, draw, thick, fill=white,
      minimum size=5mm, label=#1:#2}, scale=0.75
                    ]
\node (1) [vrtx=left/]     at ( 0, 0) {0};
\node (2) [vrtx=left/]     at (-1, 1) {};
\node (3) [vrtx=right/]    at ( 1, 1) {4};
\node (4) [vrtx=right/]    at ( 1,-1) {};
\node (5) [vrtx=left/]     at (-1,-1) {};
\node (6) [vrtx=left/]     at (0,3) {1};
\node (7) [vrtx=left/]     at (1,4) {};
\node (8) [vrtx=right/]    at (-1,4) {};
\node (9) [vrtx=right/]    at (-0.5,2) {};
\node (10) [vrtx=left/]    at (0.5,2) {};
\node (11) [vrtx=left/]    at ( 3, 0) {2};
\node (12) [vrtx=left/]    at (2, 0.5) {};
\node (13) [vrtx=right/]   at ( 4, 1) {};
\node (14) [vrtx=right/]   at ( 4,-1) {};
\node (15) [vrtx=left/]    at (2,-0.5) {};
\node (16) [vrtx=left/]    at ( -3, 0) {3};
\node (17) [vrtx=left/]    at (-4, 1) {};
\node (18) [vrtx=right/]   at ( -2, 0.5) {};
\node (19) [vrtx=right/]   at ( -4,-1) {};
\node (20) [vrtx=left/]    at (-2,-0.5) {};
\path   (1) edge (2)
        (1) edge (3)
        (1) edge (4)
        (1) edge (5)
        (6) edge (7)
        (6) edge (8)
        (6) edge (9)
        (6) edge (10)
        (11) edge (12)
        (11) edge (13)
        (11) edge (14)
        (11) edge (15)
        (16) edge (17)
        (16) edge (18)
        (16) edge (19)
        (16) edge (20)
        (1) edge (6)
        (1) edge (11)
        (1) edge (16);

\path[loosely dotted] (3) edge (7)
(4) edge (10)
(5) edge (9)
(2) edge (8)
(2) edge (12)
(3) edge (13)
(4) edge (14)
(5) edge (15)
(2) edge (17)
(3) edge (18)
(5) edge (19)
(4) edge (20)
;    
\end{tikzpicture}
\end{center}

    \item Next, our goal is to make $(w_1)_{v_0}$ a $b$-vertex. The vertices adjacent to $(w_1)_{v_0}$ are of the form $(w_1)_{v_j}$ such that $j\neq0$ and $(w_0)_{v_0}$.
    We can assign each $(w_1)_{v_j}$ the color $(j\mod m) + 1$ such that $j\in[m]$.
    Thus, $(w_1)_{v_0}$ is adjacent to a vertex of every other color.

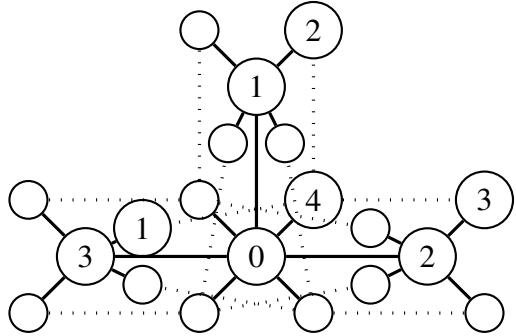
\begin{center}
\begin{tikzpicture}[
every edge/.style = {draw=black,very thick},
 vrtx/.style args = {#1/#2}{%
      circle, draw, thick, fill=white,
      minimum size=5mm, label=#1:#2}, scale=0.75
                    ]
\node (1) [vrtx=left/]     at ( 0, 0) {0};
\node (2) [vrtx=left/]     at (-1, 1) {};
\node (3) [vrtx=right/]    at ( 1, 1) {4};
\node (4) [vrtx=right/]    at ( 1,-1) {};
\node (5) [vrtx=left/]     at (-1,-1) {};
\node (6) [vrtx=left/]     at (0,3) {1};
\node (7) [vrtx=left/]     at (1,4) {2};
\node (8) [vrtx=right/]    at (-1,4) {};
\node (9) [vrtx=right/]    at (-0.5,2) {};
\node (10) [vrtx=left/]    at (0.5,2) {};
\node (11) [vrtx=left/]    at ( 3, 0) {2};
\node (12) [vrtx=left/]    at (2, 0.5) {};
\node (13) [vrtx=right/]   at ( 4, 1) {3};
\node (14) [vrtx=right/]   at ( 4,-1) {};
\node (15) [vrtx=left/]    at (2,-0.5) {};
\node (16) [vrtx=left/]    at ( -3, 0) {3};
\node (17) [vrtx=left/]    at (-4, 1) {};
\node (18) [vrtx=right/]   at ( -2, 0.5) {1};
\node (19) [vrtx=right/]   at ( -4,-1) {};
\node (20) [vrtx=left/]    at (-2,-0.5) {};
\path   (1) edge (2)
        (1) edge (3)
        (1) edge (4)
        (1) edge (5)
        (6) edge (7)
        (6) edge (8)
        (6) edge (9)
        (6) edge (10)
        (11) edge (12)
        (11) edge (13)
        (11) edge (14)
        (11) edge (15)
        (16) edge (17)
        (16) edge (18)
        (16) edge (19)
        (16) edge (20)
        (1) edge (6)
        (1) edge (11)
        (1) edge (16);

\path[loosely dotted] (3) edge (7)
(4) edge (10)
(5) edge (9)
(2) edge (8)
(2) edge (12)
(3) edge (13)
(4) edge (14)
(5) edge (15)
(2) edge (17)
(3) edge (18)
(5) edge (19)
(4) edge (20)
;    
\end{tikzpicture}
\end{center}

    \item Next, our goal is to make each $(w_0)_{v_j}$ a $b$-vertex, s.t. $j\in[m]$.
    Note that $(w_0)_{v_j}$ is already adjacent to colors 0 and $(j\mod m)+1$. 
    For the remaining $n-1$ vertices in each $(S_n)_{v_j}$, we must cycle through the rest of the colors such that we avoid $(j\mod m) + 1$, $j$, and 0. To do this, we go through each $(S_n)_{v_j}$ such that $j\neq0$ and assign a vertex $(w_i)_{v_j}$ the minimum color required to make $(w_0)_{v_j}$ a $b$-vertex. Repeat this $m-1$ times. Recall that $(w_0)_{v_j}$ is already adjacent to 2 colors. So, the next $m-1$ iterations of adding minimum colors will add $2+(m-1) = m+1$ colors to each $(S_n)_{v_j}$ such that $j\neq0$. Therefore, the only color missing in each star is our last color, $m+1$. By Lemma~\ref{lem_box_prod}, we can assign the color $m+1$ to a vertex in each $(S_n)_{v_j}$ such that $j\neq0$.
    
    Because $m\leq n$, this may not color \textit{all} of the vertices in $(S_n)_{v_j}$.
    The excess non-colored vertices after the vertex $(w_{m+1})_{v_j}$ are of the form $(w_i)_{v_j}$ such that $i\in[m+2,n]$. These excess vertices can be given the color 0. The excess vertices in $(S_n)_{v_0}$ that correspond to the excess vertices in each $(S_n)_{v_j}$, such that $j\neq0$, can be given the color 1. We know this is possible by Lemma~\ref{lem_box_prod}. If $(w_1)_{v_0}$ has the color $m+1$ and it is connected to all colors $i\in[0,m]$, then it is a $b$-vertex. If each $(w_0)_{v_j}$ is now adjacent to $m+1$ different colors, then each one is a $b$-vertex. Therefore, each possible $b$-vertex from Lemma~\ref{lem_box_prod} is colored a different color and adjacent to every other color. 

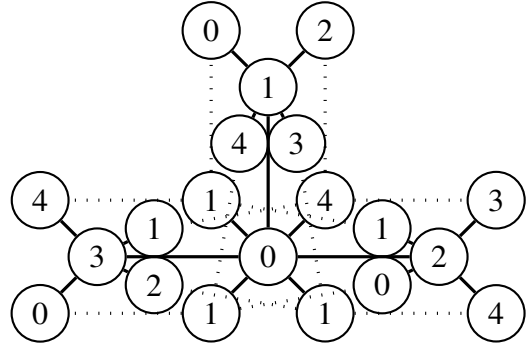
\begin{center}
\begin{tikzpicture}[
every edge/.style = {draw=black,very thick},
 vrtx/.style args = {#1/#2}{%
      circle, draw, thick, fill=white,
      minimum size=5mm, label=#1:#2}, scale=0.75
                    ]
\node (1) [vrtx=left/]     at ( 0, 0) {0};
\node (2) [vrtx=left/]     at (-1, 1) {1};
\node (3) [vrtx=right/]    at ( 1, 1) {4};
\node (4) [vrtx=right/]    at ( 1,-1) {1};
\node (5) [vrtx=left/]     at (-1,-1) {1};
\node (6) [vrtx=left/]     at (0,3) {1};
\node (7) [vrtx=left/]     at (1,4) {2};
\node (8) [vrtx=right/]    at (-1,4) {0};
\node (9) [vrtx=right/]    at (-0.5,2) {4};
\node (10) [vrtx=left/]    at (0.5,2) {3};
\node (11) [vrtx=left/]    at ( 3, 0) {2};
\node (12) [vrtx=left/]    at (2, 0.5) {1};
\node (13) [vrtx=right/]   at ( 4, 1) {3};
\node (14) [vrtx=right/]   at ( 4,-1) {4};
\node (15) [vrtx=left/]    at (2,-0.5) {0};
\node (16) [vrtx=left/]    at ( -3, 0) {3};
\node (17) [vrtx=left/]    at (-4, 1) {4};
\node (18) [vrtx=right/]   at ( -2, 0.5) {1};
\node (19) [vrtx=right/]   at ( -4,-1) {0};
\node (20) [vrtx=left/]    at (-2,-0.5) {2};
\path   (1) edge (2)
        (1) edge (3)
        (1) edge (4)
        (1) edge (5)
        (6) edge (7)
        (6) edge (8)
        (6) edge (9)
        (6) edge (10)
        (11) edge (12)
        (11) edge (13)
        (11) edge (14)
        (11) edge (15)
        (16) edge (17)
        (16) edge (18)
        (16) edge (19)
        (16) edge (20)
        (1) edge (6)
        (1) edge (11)
        (1) edge (16);

\path[loosely dotted] (3) edge (7)
(4) edge (10)
(5) edge (9)
(2) edge (8)
(2) edge (12)
(3) edge (13)
(4) edge (14)
(5) edge (15)
(2) edge (17)
(3) edge (18)
(5) edge (19)
(4) edge (20)
;    
\end{tikzpicture}
\end{center}
\end{enumerate}

Now we have shown that through this algorithm it is possible to color $S_n\square S_m$ using $m+2$ colors. By Lemma~\ref{lem_box_prod}, $m(S_n\square S_m)=m+2$ and we know $\varphi(G)\leq m(G)$. Therefore, $\varphi(S_n\square S_m)=m+2$.
\end{proof}
\begin{thm}\label{thm:line_chrom}
\textit{For two graphs $S_n$ and $S_m$,} $\varphi(L(S_n\square S_m))=m+n$
\end{thm}

First, we will prove the following lemma: that the $m$-degree is $m+n$. Recall that by definition, $\varphi(G)\leq m(G)$, so proving the $m$-degree is some integer will show that the $b$-chromatic number cannot go above the $m$-degree.

        \begin{lemma}\ \label{lem2}
        \textit{The $m$-degree of} $L(S_n\square S_m)$ \textit{is} $m+n$.

        Consider the edges of $S_n\square S_m$. The degree sequence of these edges will give us the degree sequence of $V\in L(S_n\square S_m)$. High degree vertices in $L(S_n\square S_m)$ are edges with a large neighborhood in the original graph. Observe that these must be the edges adjacent to $(w_0)_{v_0}$. Without loss of generality, let $n\ge m$, the degrees of these high degree vertices are either $2n+m-1$ or $2m+n-1$. There exist $n$ vertices of the degree $2m+n-1$ and $m$ vertices of the degree $2n+m-1$. These vertices (displayed as edges of $S_n\square S_m$) are shown in Figure~\ref{linefigure}.

\begin{figure}[h]
\centering
\begin{subfigure}[b]{0.4\textwidth}\centering
\begin{tikzpicture}[
every edge/.style = {draw=black,very thick},
 vrtx/.style args = {#1/#2}{%
      circle, draw, thick, fill=white,
      minimum size=5mm, label=#1:#2}, scale=0.75
                    ]
\node (1) [vrtx=left/]     at ( 0, 0) {};
\node (2) [vrtx=left/]     at (-1, 1) {};
\node (3) [vrtx=right/]    at ( 1, 1) {};
\node (4) [vrtx=right/]    at ( 1,-1) {};
\node (5) [vrtx=left/]     at (-1,-1) {};
\node (6) [vrtx=left/]     at (0,3) {};
\node (7) [vrtx=left/]     at (1,4) {};
\node (8) [vrtx=right/]    at (-1,4) {};
\node (9) [vrtx=right/]    at (-0.5,2) {};
\node (10) [vrtx=left/]    at (0.5,2) {};
\node (11) [vrtx=left/]    at ( 3, 0) {};
\node (12) [vrtx=left/]    at (2, 0.5) {};
\node (13) [vrtx=right/]   at ( 4, 1) {};
\node (14) [vrtx=right/]   at ( 4,-1) {};
\node (15) [vrtx=left/]    at (2,-0.5) {};
\node (16) [vrtx=left/]    at ( -3, 0) {};
\node (17) [vrtx=left/]    at (-4, 1) {};
\node (18) [vrtx=right/]   at ( -2, 0.5) {};
\node (19) [vrtx=right/]   at ( -4,-1) {};
\node (20) [vrtx=left/]    at (-2,-0.5) {};
\path   

        (1) edge (2) 
        (1) edge (3)
        (1) edge (4)
        (1) edge (5)
        (1) edge (6)
        (1) edge (11)
        (1) edge (16);

\path[loosely dotted] (3) edge (7)
(4) edge (10)
(5) edge (9)
(2) edge (8)
(2) edge (12)
(3) edge (13)
(4) edge (14)
(5) edge (15)
(2) edge (17)
(3) edge (18)
(5) edge (19)
(4) edge (20)
(6) edge (7)
        (6) edge (8)
        (6) edge (9)
        (6) edge (10)
        (11) edge (12)
        (11) edge (13)
        (11) edge (14)
        (11) edge (15)
        (16) edge (17)
        (16) edge (18)
        (16) edge (19)
        (16) edge (20);
\end{tikzpicture}
\end{subfigure}
\caption{$S_4\square S_3$}
\label{linefigure}
\end{figure}
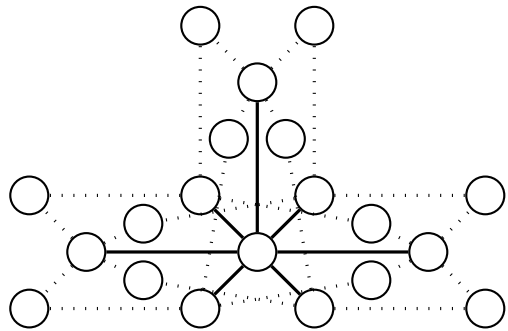

        We will consider the subset of $V(G)$, $\{\{(w_i)_{v_0},(w_0)_{v_0}\}\in E(G):i\neq0\}$, as the set $U$. Each element in $U$ has the degree $2m+n-1$. Then we will consider the subset of $V(L(S_n\square S_m))$, $\{\{(w_0)_{v_j},(w_0)_{v_0}\}\in E(G):j\neq0\}$, as the set $W$. Each element in $W$ has the degree $2n+m-1$.
        Now we will show that elements in set $W$ have a higher degree than the vertices in the set $U$. If $n\ge m$, then $(n+m) + n\ge (n+m)+n$ which simplifies to $2n+m\ge2m+n$. By additivity, $2n+m-1\ge 2m+n-1$. Therefore, we know that all $m$ and $n$ vertices of degree $2n+m-1$ and $2m+n-1$ respectively, must both be at \textit{least} $2m+n-1$ degree. So there exist $m+n$ vertices of degree $2m+n-1$. This shows that each element in $U\cup W$ has a degree of \textit{at least} $2m+n-1$. Observe that $|U\cup W|=m+n$ and $m+n-1< 2m+n-1$, therefore the $m$-degree is at least $m+n$. The $m$-degree cannot be higher than $m+n$ as all of the vertices that exist in $(U\cup W)^c$ have a degree lower than $m+n$.
        Therefore, by the definition of the $m$-degree, there exist $m+n$ vertices of at least $(m+n-1)$-degree so, $m(L(S_n\square S_m))=m+n$.

         \end{lemma}
        \begin{proof}[Proof of Theorem~\ref{thm:line_chrom}]\

       Observe that $U\cup W$ forms a subset of $S_n\square S_m$,
        
        $$\bigcup_{i=1}^n\{(w_i)_{v_0},(w_0)_{v_0}\}\cup\bigcup_{j=1}^m\{(w_0)_{v_j},(w_0)_{v_0}\}$$

         If each element in the set $U\cup W$ contains the $(w_0)_{v_0}$ in their edge pairs, then they form a clique of size $|U\cup W|$. Lemma~\ref{lem2} shows that $|U\cup W|$ is $m+n$. Therefore, by Proposition~\ref{squeeze},  $\varphi(L(S_n\square S_m))=m+n$. 
    
\end{proof}

\begin{rem}
    It should be noted that $\omega(L(S_n\square S_m))$ can alternatively be shown by the well known isomorphism, \begin{equation}\label{rookiso}
        L(K_{i,j})\cong K_i\square K_j
        \end{equation} For $i=1$ and $j=m+n$, we have $L(K_{1,m+n}) \cong K_{m+n}\square K_1\cong K_{m+n}$. The isomorphism shown in (\ref{rookiso}) defines the rook graph. It is interesting to note that the rook graph has several relationships with the star graph and its operations. We will look further into the rook graph in Theorem \ref{thm4.2} and Theorem \ref{rookgraphcor}.
\end{rem}

\begin{thm}\label{thm:line_tot}
\textit{For $n\ge m\ge 3$, It follows that,} $$\varphi(T(S_n\square S_m))=\begin{cases}
    2m+n+1, & \textit{If $n> 2(m-1)$}\\
    2n+3 & \textit{If $n\leq 2(m-1)$}
\end{cases}$$
\end{thm}

It should be noted that at $n=2(m-1)$, $2m+n+1=2n+3$.

First, we will prove the $m$-degree of $T(S_n\square S_m)$. Recall that $\varphi(G)\leq m(G)$.

\begin{lemma} \textit{If $n\ge m\ge 3$, then the following are true about the $m$-degree of $T(S_n\square S_m)$,}
\begin{enumerate}
    \item If $n> 2(m-1)$, then $ m(T(S_n\square S_m))=2m+n+1$.
    \item If $n\leq 2(m-1)$, then $ m(T(S_n\square S_m))=2n+3$.
\end{enumerate}

We'll first partition the vertices of $T(S_n\square S_m)$ into disjoint sets satisfying the following properties.

  Let $\alpha_0\in X_0=\left\{(w_0)_{v_0}\right\}$ . Then, $|X_0|=1$ and $\deg(\alpha_0)=2n+2m$. 

  Let $\alpha_1\in X_1=\left\{\{(w_0)_{v_0},(w_0)_{v_j}\}:j\neq0\right\}$. Then, $|X_1|=m$ and $\deg(\alpha_1)=2n+m+1$.

 Let $\alpha_2\in X_2=\left\{(w_0)_{v_j}:j\neq0\right\}$. Then,  $|X_2|=m$ and $\deg(\alpha_2)=2n+2$.

   Let $\alpha_3\in X_3=\left\{\{(w_0)_{v_0},(w_i)_{v_0}\}:i\neq0\right\}$. Then, $|X_3|=n$ and $\deg(\alpha_3)=2m+n+1$.

   Let $\alpha_4\in X_4=\left\{(w_i)_{v_0}:i\neq0\right\}$. Then, $|X_4|=n$ and $\deg(\alpha_4)=2m+2$.

   Let $\alpha_5\in X_5=\left\{\{(w_0)_{v_j},(w_i)_{v_j}\}:i,j\neq0\right\}$. Then, $|X_5|=nm$ and $\deg(\alpha_5)=n+3$.

    Let $\alpha_6\in X_6=\left\{\{(w_i)_{v_{j_1}},(w_i)_{v_{j_2}}\}:j_1\neq j_2\neq0\right\}$. Then, $|X_6|=nm$ and $\deg(\alpha_6)=m+3$.

   Let $\alpha_7\in X_7=\left\{(w_i)_{v_j}:i,j\neq0\right\}$. Then,  $|X_7|=nm$ and $\deg(\alpha_7)=4$.

  These subsets are distinct and the following is true, $$\bigsqcup_{i=0}^7X_i = V(T(S_n\square S_m))$$

\textbf{Case 1:} Let $n> 2(m-1)$.

  Then, $2n > 2m+n-2$. It follows that $2n+2>2m+n$ and then $2n+2\ge 2m+n + 1$. Thus,
  
  $$deg(\alpha_0)\ge deg(\alpha_1)\ge deg(\alpha_2)\ge deg(\alpha_3)\ge deg(\alpha_4)\ge deg(\alpha_5)\ge deg(\alpha_6)\ge deg(\alpha_7)$$

Observe that there are $2m+n+1$ vertices in $X_0\cup X_1\cup X_2\cup X_3$, and each vertex has degree at least $2m+n$.
The vertices in $X_4$ through $X_7$ have degree no greater than $2m+2$. There are no more vertices other than the $|X_0\cup X_1\cup X_2\cup X_3|$ elements of degree greater than or equal to $2m+n$. Therefore, the $m$-degree is $2m+n+1$.

\textbf{Case 2:} Let $n\leq2(m-1)$.

We will show that the vertices in $X_3$ have degree at least $2n+2$.
Hence, we must show that $2n+2 \le 2m+n+1$.

If $n\leq2m-2$, then $2n\leq 2m+n-2$. Then, $2n+2\leq 2m+n< 2m+n+1$. Therefore, if $n\leq2(m-1)$, then $2n+2< 2m+n+1$.

Next, we show that the cardinality of the sets $X_0, X_1, X_2, X_3$ is greater than or equal to $2n+3$.
In other words, we must show that $1+m+m+n\ge 2n+3$.
This was proven above.

Now that we understand that the vertices in the set $X_3$ have a greater degree than the vertices in the set $X_2$, we can follow a similar exercise that we did in case 1.

Suppose that the $m$-degree is $2n+4$. By the definition of the $m$-degree, there must exist $2n+4$ vertices of degree $2n+3$. There exist $n+m+1$ vertices from $X_0,X_1,X_3$ of at least degree $2n+2$.
However, $n+m+1< 2n+1 <2n+2$, and no other set $X_i$ has vertices with degree greater than or equal to $2n+2$.
Thus, the $m$-degree of $T(S_n\square S_m) \not \ge 2n+4$.
Therefore $m(T(S_n\square S_m))= 2n+3$.

\end{lemma}

\begin{proof}[Proof of Theorem~\ref{thm:line_tot}] \

We will now show by conditional cases, we can color the graph $T(S_n\square S_m)$ either with $2m+n+1$ colors or with $2n+3$ colors.

\textbf{Case 1: } Let $n>2(m-1)$

The vertex in $X_0$ is $(w_0)_{v_0}$. 
The vertices in $X_1$ are the edges of the form $\{(w_0)_{v_0},(w_0)_{v_j}\}$ such that $j\in[m]$.
The vertices in $X_2$ are $(w_0)_{v_j}$ themselves.
Finally, the vertices in $X_3$ are $(\{w_0,w_i\})_{v_0}$ such that $i\in[n]$. It is worth noting that the subsets $X_0, X_1,X_3$ form a complete sub-graph of order $m+n+1$. Using this information, we can now describe an algorithm that will color the graph using $2m+n+1$ colors.

\begin{enumerate}
    \item First, assign $(w_0)_{v_0}$ the color 1. Then assign each edge adjacent to $(w_0)_{v_0}$ colors from 2 to $m+n+1$. By the same reasoning of [\ref{thm:line_chrom}], the $b$-chromatic number of this sub-graph is $m+n+1$.

\item Next, assign the vertices $(w_0)_{v_j}$, such that $j\in[m]$, the colors $c\in[m+n+2,2m+n+1]$.

\item The edge $\{(w_0)_{v_0},(w_0)_{v_j}\}$ is already adjacent to the colors assigned to the edges adjacent to $(w_0)_{v_0}$ and to the color 1. But it is not adjacent to the colors of the vertices $(w_0)_{v_i}$ such that $i\neq j$. Color any $m-1$ edges in the inner star $(S_n)_{v_j}$ with these missing colors. This makes the edges $\{(w_0)_{v_0},(w_0)_{v_j}\}$ $b$-vertices.

\item  The vertex $(w_0)_{v_j}$ is already adjacent to $m+1$ colors, therefore we color the remaining edges $(\{w_0,w_i\})_{v_j}$, such that $j\neq0$ with $n-(m-1)$ additional colors. In total, $(w_0)_{v_j}$ is adjacent to $n+2$ colors. Now color each $(w_i)_{v_j}$ such that $i,j\neq 0$, with $2m-1$ remaining colors. This is possible as $2m-1\leq n$. The coloring for the  vertices $(w_i)_{v_j}$ such that $i,j\neq0$ doesn't matter as they are only adjacent to $(w_0)_{v_j}$ and $(w_i)_{v_0}$. We haven't colored these. Therefore, no matter what order we choose to color in, we will have a proper coloring. This coloring makes the vertices $(w_0)_{v_j}$ $b$-vertices. Any remaining vertices of the form $(w_i)_{v_j}$ for $i,j\neq0$ can be given the color 1.

\item
Each $\left\{(w_0)_{v_0},(w_i)_{v_0}\right\}$ is not a $b$-vertex as they are not adjacent to the $m$ colors of each $(w_0)_{v_j}$ such that $j\neq0$.
Each edge in $(S_n)_{v_0}$ corresponds to $m$ unique edges of the form $\left\{(w_i)_{v_0},(w_i)_{v_j}\right\}$ such that $j\in[m]$. Assign the edges $\left\{(w_i)_{v_0},(w_i)_{v_j}\right\}$ the colors $c((w_0)_{v_j})$ for $j\in[m]$.
This makes each $\left\{(w_0)_{v_0},(w_i)_{v_0}\right\}$ a $b$-vertex.

\item The rest of the vertices are of low degree. The vertices adjacent and the edges incident to these vertices are already colored. So we color these vertices in a way that ensures the graph is properly colored. We have at least 6 additional colors to choose from.

\end{enumerate}

\textbf{Case 2: } Let $n\leq 2(m-1)$.

We first define the following mapping,

$$\phi_j: E\left((S_n)_{v_0}\right)\xrightarrow{} V\left((S_n)_{v_j}\right)\setminus\left\{(w_0)_{v_j}\right\}, \text{ where } \phi_j\left(\left\{w_0,w_i\right\}_{v_0}\right) = (w_i)_{v_j}$$

In other words, $\phi$ maps an edge in the central inner star to one non-central vertex in each non-central inner star. Observe that $\phi$ is a bijection, so therefore must be invertible.

Our algorithm will start similarly to the former.

\begin{enumerate}
    \item First, assign $(w_0)_{v_0}$ the color 1. Then assign each $(\{w_0,w_i\})_{v_0}$ , a color $i\in[2,n+1]$.

    \item Then, assign each $(w_0)_{v_j}$ a color $j\in[n+2, n+m+1]$.

    \item Then, choose the edges $\left\{(w_0)_{v_0},(w_0)_{v_j}\right\}$ where $j\in [n-m+2]$ and assign them colors $c\in[n+m+2, 2n+3]$.
    For us to be able to choose $n-m+2$ edges, it must be true that $m\geq n-m+2$. By the condition of this case, this is true.

We have now established what our $b$-vertices will be. Steps 4 through 8 of this algorithm will make sure these vertices are $b$-vertices.

\item Observe that $(w_0)_{v_0}$ is a $b$-vertex. Next, we will color the edges and vertices adjacent to each $(w_0)_{v_j}$ with certain colors such that each $(w_0)_{v_j}$ becomes a $b$-vertex.
Let $e_i = (\{w_0, w_i\})_{v_0}$. Then, assign $\phi_j(e_i)$ the color $c(e_i)$ for $i\in[n]$.

\item Let $j\in [n-m+2]$. 
Each vertex $(w_0)_{v_j}$ is missing the colors $c((w_0)_{v_k})$ for $k\ne 0, j$, and the colors $c(\{(w_0)_{v_0},(w_0)_{v_{k'}}\})$ for $k'\in [n-m+2]\setminus \{j\}$.
So, we need $(m-1)+(n-m+1)=n$ additional colors.
We will use the $n$ uncolored edges of the form $\{w_0, w_i\}_{v_j}$ to fill in these missing colors.
We can color these edges in any order and maintain a proper coloring.

\item Now we make the remaining vertices $(w_{0})_{v_j}$ into $b$-vertices.
These vertices are also missing the $(m-1)$ colors  $c((w_0)_{v_k})$ where $k\neq j, 0$.
In addition, they are missing all of the $n-m+2$ colors $c(\{(w_0)_{v_0},(w_0)_{v_{k'}}\})$, where $k'\in [n-m+2]$.
We begin by coloring $n-m+2$ of the edges $\{w_0, w_i\}_{v_j}$ with the colors $c(\{(w_0)_{v_0},(w_0)_{v_{k'}}\})$, where $k'\in [n-m+2]$.
(We cannot color the edge $\{(w_0)_{v_0}, (w_0)_{v_j}\}$ with any of these $n-m+2$ colors, because this edge is adjacent to the edges $\{(w_0)_{v_0}, (w_0)_{v_{k'}}\}$.)
Next we color the remaining  $n-(n-m+2)$ edges $\{w_0, w_i\}_{v_j}$ and the edge $\{(w_0)_{v_0}, (w_0)_{v_j}\}$ with the colors $c((w_0)_{v_k})$ where $k\neq j, 0$, in any order.
We have exactly $m-1$ edges to color with exactly $m-1$ missing colors.
At the conclusion of this step, all of the vertices $(w_{0})_{v_j}$, where $j\in [m]$, are $b$-vertices.

\item Each $\left\{(w_0)_{v_0},(w_i)_{v_0}\right\}$ is not a $b$-vertex as they are not adjacent to the $m$ colors of each $(w_0)_{v_j}$ such that $j\neq0$.
Each edge in $(S_n)_{v_0}$ corresponds to $m$ unique edges of the form $\left\{(w_i)_{v_0},(w_i)_{v_j}\right\}$ such that $j\in[m]$. Assign the edges $\left\{(w_i)_{v_0},(w_i)_{v_j}\right\}$ the colors $c((w_0)_{v_j})$ for $j\in[m]$.
This makes each $\left\{(w_0)_{v_0},(w_i)_{v_0}\right\}$ a $b$-vertex. 

\item Now, we have $2n+3$ vertices of different color classes adjacent to every other color class. All we need now is a proper coloring. If a vertex has a neighborhood of vertices, $\{\alpha_1,\ldots,\alpha_k\}$, then assign the vertex a color $c\notin\{\alpha_1,\ldots,\alpha_k\}:c\leq 2n+3$. To show this is possible, observe that only uncolored elements in $T(S_n\square S_m)$ are elements in $X_4\cup X_5\cup X_6\cup X_7$, with the largest neighborhood of any vertex in this subset being $2m+2$. As such, the $\max\{\alpha_1,\ldots,\alpha_k\}=2m+2$ and because $2m+2<2n+3$, we can always choose a color and maintain a proper coloring. 
\end{enumerate}

Therefore, by our algorithms, if $n>2(m-1)$ then, $\varphi(T(S_n\square S_m))=2n+m+1$ and if $n\leq 2(m-1)$, then $\varphi(T(S_n\square S_m))=2n+3$.
\end{proof}

\section{Power Graph of Star Graphs}\label{sec:powers}

\begin{fact}\label{fact 1}
    For any graph $G$ of order $n$, if $Diam(G)\le p$, then $\varphi(G^p) = n$, with $p\ge2$.
\end{fact}
\begin{proof}
    Suppose $Diam(G)\le p$, it is trivial to see that $G^p$ is a complete graph. So $\varphi(G^p) = n$.
\end{proof}

\begin{thm} \label{power1simple}
   \textit{ Let $G$ be the $k$-graph power of $S_n$. It follows that,} $$\varphi(G)=\begin{cases}
       2, & \textit{If $k=1$}\\
       n+1, & \textit{If $k>1$}
   \end{cases}$$
\end{thm}

\begin{proof}\
  Observe $\varphi(S_n^1)=2$. For all $n>1, Diam(S_n)=2$, by Fact~\ref{fact 1}, if $k>1$, then $\varphi\left(S_n^k\right)=n+1$.
\end{proof}

We will prove the following useful lemmas for the second and third parts of Theorem \ref{thm4.2}.
\begin{lemma} \label{lemma 5}
    
\textit{The $m$-degree of $\left(S_n\square S_m\right)^2$ is $m+n+2$}.

\end{lemma}

\begin{proof}

High degree vertices in this graph are now $(w_0)_{v_0}$ and vertices with a distance of one from $(w_0)_{v_0}$.
The only vertices that are distance 2 from $(w_0)_{v_0}$ are all $(w_i)_{v_j}$ such that $i,j\neq0$.
There are $nm$ of these vertices, so therefore after the graph power, $(w_0)_{v_0}$ must have a degree of $m + n + (nm)$.
Each $(w_0)_{v_j}$, such that $j\neq0$, originally has degree $n+1$.
The vertices that are distance 2 from $(w_0)_{v_j}$ are the other $m-1$ vertices $(w_0)_{v_k}$ and the $n$ vertices $(w_i)_{v_0}$ such that $i\in[n]$. 
After the graph power, $deg((w_0)_{v_j})=n+1 + (m-1) + n = 2n + m$.
Next, each $(w_i)_{v_0}$, such that $i\neq0$, originally has degree $m+1$.
The vertices that are distance 2 from $(w_i)_{v_0}$ are:
    \begin{itemize}
    \item each of the other $n-1$ vertices $(w_i)_{v_0}$ such that $i\neq0$ and 
    \item each $(w_0)_{v_j}$ such that $j\neq0$. 
    \end{itemize}
    Therefore, after the graph power, the degree of each vertex $(w_i)_{v_0}$ such that $i\neq0$ is $m+1 + (n-1) + m = 2m+n$. The size of this set of high-degree vertices is $m+n+1$, and consequently, there exist $m+n+1$ vertices of at least degree $m+n$.
    To check for a higher $m$-degree, we can consider the vertices $(w_i)_{v_j}$ such that $i,j\neq0$. Before the graph power, these vertices have a degree of $2$. 
    Then, after the graph power $(w_i)_{v_j}$ is adjacent to:
    \begin{itemize}
        \item $(w_0)_{v_0}$;
        \item each of the $n-1$ vertices $(w_k)_{v_j}$ where $j$ is fixed and $k\neq i,0$;  
        \item each of the $m-1$ vertices $(w_i)_{v_k}$ where $i$ is fixed and $k \neq j, 0$.
\end{itemize}
 In total, the degrees of these vertices are $2 + (n-1) + (1) + (m-1) = m+n+1$.
To summarize these claims, we have 1 vertex of degree $nm+n+m$, $m$ vertices of degree $2n+m$, $n$ vertices of degree $2m+n$, and $nm$ vertices of degree $n+m+1$.
Thus, all vertices have degree at least $n+m+1$.
There are $(n+1)(m+1) = nm+n+m+1$, and clearly $nm+n+m+1 > n+m+2$ because $n,m\ge 1$.
If we consider any higher $m$-degree, we need at least $m+n+3$ vertices of degree strictly greater than $m+n+1$, which we do not have (because we only have $m+n+1$ vertices with degree greater than $m+n+3$).
Therefore, $n+m+2$ is the highest $m$-degree achievable.

\end{proof}

\begin{lemma}\label{Lemma 6} 
    
\textit{For $n,m\ge2$, the $m$-degree of $\left(S_n\square S_m\right)^3$ is $2m+2n$.}

\end{lemma}

\begin{proof}

In the second graph power, the each vertex of the form $(w_0)_{v_j}$ is adjacent to every vertex except $(w_i)_{v_k}$ where $k\neq j,0$, and they are distance one apart. 
Thus, they are now adjacent in the third graph power. 
Similarly, in the second graph power, each vertex $(w_i)_{v_0}$ is adjacent to every vertex except for the leaves $(w_l)_{v_k}$ where $l\neq i$ and $k\neq0$.
Again, these vertices are now adjacent in the third graph power. 

Thus, in the third graph power, the only vertices that are not adjacent to every other vertex are the leaves $(w_{i})_{v_{j}}$, where $i,j \neq 0$.
In particular, $(w_{i})_{v_{j}}\not\sim(w_{k})_{v_{l}}$ iff $i\neq k$, $j\neq l$, and $i,j,k,l\neq0$.

This implies that $(w_0)_{v_0}$ still has a degree of $nm+n+m$. Additionally, $(w_0)_{v_j}$ such that $j\neq0$ and $(w_i)_{v_0}$ such that $i\neq0$ will all have a degree of $nm + n + m$. (This can be seen in Figure~\ref{power3}, where the graph is organized by degree, where the highest degree vertices are closest to the top and the lowest degree vertices are near the bottom.)

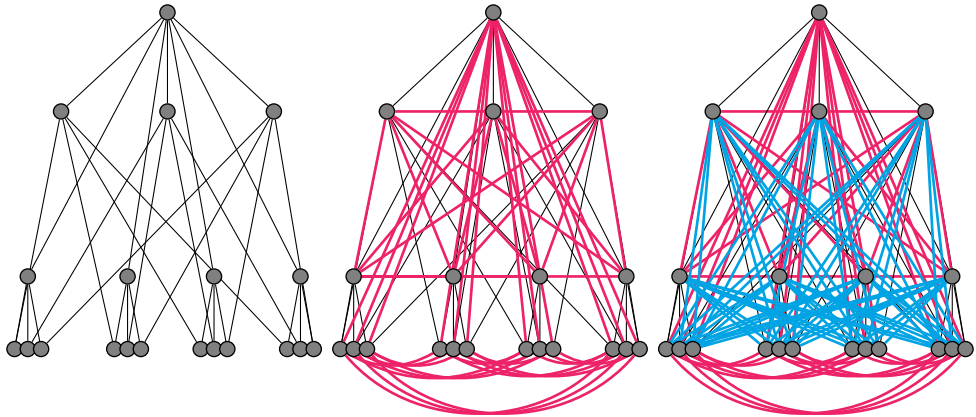
\begin{figure}[h]
\centering

\begin{tikzpicture}[scale=0.35]
	\begin{pgfonlayer}{nodelayer}
		\node [style=White] (0) at (0, 12.75) {};
		\node [style=White] (1) at (-4, 9) {};
		\node [style=White] (2) at (0, 9) {};
		\node [style=White] (3) at (4, 9) {};
		\node [style=White] (4) at (-5.25, 2.75) {};
		\node [style=White] (5) at (-1.5, 2.75) {};
		\node [style=White] (6) at (1.75, 2.75) {};
		\node [style=White] (7) at (5, 2.75) {};
		\node [style=White] (8) at (-5.75, 0) {};
		\node [style=White] (9) at (-5.25, 0) {};
		\node [style=White] (10) at (-4.75, 0) {};
		\node [style=White] (11) at (-2,0) {};
		\node [style=White] (12) at (-1.5, 0) {};
		\node [style=White] (13) at (-1, 0) {};
		\node [style=White] (14) at (1.25, 0) {};
		\node [style=White] (15) at (1.75, 0) {};
		\node [style=White] (16) at (2.25, 0) {};
		\node [style=White] (17) at (4.5, 0) {};
		\node [style=White] (18) at (5, 0) {};
		\node [style=White] (19) at (5.5, 0) {};
		\node [style=White] (20) at (12.25, 12.75) {};
		\node [style=White] (21) at (8.25, 9) {};
		\node [style=White] (22) at (12.25, 9) {};
		\node [style=White] (23) at (16.25, 9) {};
		\node [style=White] (24) at (7, 2.75) {};
		\node [style=White] (25) at (10.75, 2.75) {};
		\node [style=White] (26) at (14, 2.75) {};
		\node [style=White] (27) at (17.25, 2.75) {};
		\node [style=White] (28) at (6.5, 0) {};
		\node [style=White] (29) at (7, 0) {};
		\node [style=White] (30) at (7.5, 0) {};
		\node [style=White] (31) at (10.25, 0) {};
		\node [style=White] (32) at (10.75, 0) {};
		\node [style=White] (33) at (11.25, 0) {};
		\node [style=White] (34) at (13.5, 0) {};
		\node [style=White] (35) at (14, 0) {};
		\node [style=White] (36) at (14.5, 0) {};
		\node [style=White] (37) at (16.75, 0) {};
		\node [style=White] (38) at (17.25, 0) {};
		\node [style=White] (39) at (17.75, 0) {};
		\node [style=White] (40) at (24.5, 12.75) {};
		\node [style=White] (41) at (20.5, 9) {};
		\node [style=White] (42) at (24.5, 9) {};
		\node [style=White] (43) at (28.5, 9) {};
		\node [style=White] (44) at (19.25, 2.75) {};
		\node [style=White] (45) at (23, 2.75) {};
		\node [style=White] (46) at (26.25, 2.75) {};
		\node [style=White] (47) at (29.5, 2.75) {};
		\node [style=White] (48) at (18.75, 0) {};
		\node [style=White] (49) at (19.25, 0) {};
		\node [style=White] (50) at (19.75, 0) {};
		\node [style=White] (51) at (22.5, 0) {};
		\node [style=White] (52) at (23, 0) {};
		\node [style=White] (53) at (23.5, 0) {};
		\node [style=White] (54) at (25.75, 0) {};
		\node [style=White] (55) at (26.25, 0) {};
		\node [style=White] (56) at (26.75, 0) {};
		\node [style=White] (57) at (29, 0) {};
		\node [style=White] (58) at (29.5, 0) {};
		\node [style=White] (59) at (30, 0) {};
	\end{pgfonlayer}
	\begin{pgfonlayer}{edgelayer}
		\draw (0) to (1);
		\draw (0) to (2);
		\draw (0) to (3);
		\draw (0) to (4);
		\draw (0) to (5);
		\draw (0) to (6);
		\draw (0) to (7);
		\draw (4) to (8);
		\draw (4) to (9);
		\draw (4) to (10);
		\draw (5) to (11);
		\draw (5) to (12);
		\draw (5) to (13);
		\draw (6) to (14);
		\draw (6) to (15);
		\draw (6) to (16);
		\draw (7) to (17);
		\draw (7) to (18);
		\draw (7) to (19);
		\draw (1) to (8);
		\draw (1) to (11);
		\draw (1) to (14);
		\draw (1) to (17);
		\draw (2) to (9);
		\draw (3) to (10);
		\draw (2) to (12);
		\draw (3) to (13);
		\draw (2) to (15);
		\draw (3) to (16);
		\draw (2) to (18);
		\draw (3) to (19);
		\draw (8) to (9);
		\draw (9) to (10);
		\draw (11) to (12);
		\draw (12) to (13);
		\draw (14) to (15);
		\draw (15) to (16);
		\draw (17) to (18);
		\draw (18) to (19);
		\draw (20) to (21);
		\draw (20) to (22);
		\draw (20) to (23);
		\draw (20) to (24);
		\draw (20) to (25);
		\draw (20) to (26);
		\draw (20) to (27);
		\draw (24) to (28);
		\draw (24) to (29);
		\draw (24) to (30);
		\draw (25) to (31);
		\draw (25) to (32);
		\draw (25) to (33);
		\draw (26) to (34);
		\draw (26) to (35);
		\draw (26) to (36);
		\draw (27) to (37);
		\draw (27) to (38);
		\draw (27) to (39);
		\draw (21) to (28);
		\draw (21) to (31);
		\draw (21) to (34);
		\draw (21) to (37);
		\draw (22) to (29);
		\draw (23) to (30);
		\draw (22) to (32);
		\draw (23) to (33);
		\draw (22) to (35);
		\draw (23) to (36);
		\draw (22) to (38);
		\draw (23) to (39);
            \draw [bend right, style=red edge] (28) to (31);
            \draw [bend right, style=red edge] (28) to (34);
            \draw [bend right=50, style=red edge] (28) to (37);
            \draw [bend right, style=red edge] (29) to (32);
            \draw [bend right, style=red edge] (29) to (35);
            \draw [bend right=50, style=red edge] (29) to (38);
            \draw [bend right, style=red edge] (30) to (33);
            \draw [bend right, style=red edge] (30) to (36);
            \draw [bend right=50, style=red edge] (30) to (39);
            \draw [bend right, style=red edge] (31) to (34);
            \draw [bend right, style=red edge] (32) to (35);
            \draw [bend right, style=red edge] (33) to (36);
            \draw [bend right, style=red edge] (31) to (37);
            \draw [bend right, style=red edge] (32) to (38);
            \draw [bend right, style=red edge] (33) to (39);
            \draw [bend right, style=red edge] (34) to (37);
            \draw [bend right, style=red edge] (35) to (38);
            \draw [bend right, style=red edge] (36) to (39);
            \draw [style=red edge] (21) to (22);
            \draw [style=red edge] (22) to (23);
		\draw [style=red edge] (28) to (29);
		\draw [style=red edge] (29) to (30);
		\draw [style=red edge] (31) to (32);
		\draw [style=red edge] (32) to (33);
		\draw [style=red edge] (34) to (35);
		\draw [style=red edge] (35) to (36);
		\draw [style=red edge] (37) to (38);
		\draw [style=red edge] (38) to (39);
		\draw [style=red edge] (21) to (24);
		\draw [style=red edge] (21) to (25);
		\draw [style=red edge] (21) to (26);
		\draw [style=red edge] (21) to (27);
		\draw [style=red edge] (22) to (24);
		\draw [style=red edge] (22) to (25);
		\draw [style=red edge] (22) to (26);
		\draw [style=red edge] (22) to (27);
		\draw [style=red edge] (23) to (24);
		\draw [style=red edge] (23) to (25);
		\draw [style=red edge] (23) to (26);
		\draw [style=red edge] (23) to (27);
		\draw [style=red edge] (20) to (28);
		\draw [style=red edge] (20) to (29);
		\draw [style=red edge] (20) to (30);
		\draw [style=red edge] (20) to (31);
		\draw [style=red edge] (20) to (32);
		\draw [style=red edge] (20) to (33);
		\draw [style=red edge] (20) to (34);
		\draw [style=red edge] (20) to (35);
		\draw [style=red edge] (20) to (36);
		\draw [style=red edge] (20) to (37);
		\draw [style=red edge] (20) to (38);
		\draw [style=red edge] (20) to (39);
            \draw [style=red edge] (24) to (25);
            \draw [style=red edge] (24) to (26);
            \draw [style=red edge] (24) to (27);
            \draw [style=red edge] (45) to (46);
            \draw [style=red edge] (46) to (47);
            \draw [style=red edge] (44) to (45);
            \draw [style=red edge] (42) to (43);
            \draw [style=red edge] (41) to (42);
            \draw [bend right, style=red edge] (48) to (51);
            \draw [bend right, style=red edge] (48) to (54);
            \draw [bend right=50, style=red edge] (48) to (57);
            \draw [bend right, style=red edge] (49) to (52);
            \draw [bend right, style=red edge] (49) to (55);
            \draw [bend right=50, style=red edge] (49) to (58);
            \draw [bend right, style=red edge] (50) to (53);
            \draw [bend right, style=red edge] (50) to (56);
            \draw [bend right=50, style=red edge] (50) to (59);
            \draw [bend right, style=red edge] (51) to (54);
            \draw [bend right, style=red edge] (52) to (55);
            \draw [bend right, style=red edge] (53) to (56);
            \draw [bend right, style=red edge] (51) to (57);
            \draw [bend right, style=red edge] (52) to (58);
            \draw [bend right, style=red edge] (53) to (59);
            \draw [bend right, style=red edge] (54) to (57);
            \draw [bend right, style=red edge] (55) to (58);
            \draw [bend right, style=red edge] (56) to (59);
		\draw (40) to (41);
		\draw (40) to (42);
		\draw (40) to (43);
		\draw (40) to (44);
		\draw (40) to (45);
		\draw (40) to (46);
		\draw (40) to (47);
		\draw (44) to (48);
		\draw (44) to (49);
		\draw (44) to (50);
		\draw (45) to (51);
		\draw (45) to (52);
		\draw (45) to (53);
		\draw (46) to (54);
		\draw (46) to (55);
		\draw (46) to (56);
		\draw (47) to (57);
		\draw (47) to (58);
		\draw (47) to (59);
		\draw (41) to (48);
		\draw (41) to (51);
		\draw (41) to (54);
		\draw (41) to (57);
		\draw (42) to (49);
		\draw (43) to (50);
		\draw (42) to (52);
		\draw (43) to (53);
		\draw (42) to (55);
		\draw (43) to (56);
		\draw (42) to (58);
		\draw (43) to (59);
		\draw [style=red edge] (48) to (49);
		\draw [style=red edge] (49) to (50);
		\draw [style=red edge] (51) to (52);
		\draw [style=red edge] (52) to (53);
		\draw [style=red edge] (54) to (55);
		\draw [style=red edge] (55) to (56);
		\draw [style=red edge] (57) to (58);
		\draw [style=red edge] (58) to (59);
		\draw [style=red edge] (41) to (44);
		\draw [style=red edge] (41) to (45);
		\draw [style=red edge] (41) to (46);
		\draw [style=red edge] (41) to (47);
		\draw [style=red edge] (42) to (44);
		\draw [style=red edge] (42) to (45);
		\draw [style=red edge] (42) to (46);
		\draw [style=red edge] (42) to (47);
		\draw [style=red edge] (43) to (44);
		\draw [style=red edge] (43) to (45);
		\draw [style=red edge] (43) to (46);
		\draw [style=red edge] (43) to (47);
		\draw [style=red edge] (40) to (48);
		\draw [style=red edge] (40) to (49);
		\draw [style=red edge] (40) to (50);
		\draw [style=red edge] (40) to (51);
		\draw [style=red edge] (40) to (52);
		\draw [style=red edge] (40) to (53);
		\draw [style=red edge] (40) to (54);
		\draw [style=red edge] (40) to (55);
		\draw [style=red edge] (40) to (56);
		\draw [style=red edge] (40) to (57);
		\draw [style=red edge] (40) to (58);
		\draw [style=red edge] (40) to (59);
		\draw [style=blue edge] (41) to (50);
		\draw [style=blue edge] (41) to (49);
		\draw [style=blue edge] (41) to (52);
		\draw [style=blue edge] (41) to (53);
		\draw [style=blue edge] (41) to (55);
		\draw [style=blue edge] (41) to (56);
		\draw [style=blue edge] (41) to (58);
		\draw [style=blue edge] (41) to (59);
		\draw [style=blue edge] (42) to (48);
		\draw [style=blue edge] (42) to (50);
		\draw [style=blue edge] (42) to (51);
		\draw [style=blue edge] (42) to (53);
		\draw [style=blue edge] (42) to (54);
		\draw [style=blue edge] (42) to (56);
		\draw [style=blue edge] (42) to (57);
		\draw [style=blue edge] (42) to (59);
		\draw [style=blue edge] (43) to (49);
		\draw [style=blue edge] (43) to (48);
		\draw [style=blue edge] (43) to (51);
		\draw [style=blue edge] (43) to (52);
		\draw [style=blue edge] (43) to (54);
		\draw [style=blue edge] (43) to (55);
		\draw [style=blue edge] (43) to (57);
		\draw [style=blue edge] (43) to (58);
            \draw [style=blue edge] (44) to (51);
            \draw [style=blue edge] (44) to (52);
            \draw [style=blue edge] (44) to (53);
            \draw [style=blue edge] (44) to (54);
            \draw [style=blue edge] (44) to (55);
            \draw [style=blue edge] (44) to (56);
            \draw [style=blue edge] (44) to (57);
            \draw [style=blue edge] (44) to (58);
            \draw [style=blue edge] (45) to (48);
            \draw [style=blue edge] (45) to (49);
            \draw [style=blue edge] (45) to (50);
            \draw [style=blue edge] (45) to (54);
            \draw [style=blue edge] (45) to (55);
            \draw [style=blue edge] (45) to (56);
            \draw [style=blue edge] (45) to (57);
            \draw [style=blue edge] (45) to (58);
            \draw [style=blue edge] (45) to (59);
            \draw [style=blue edge] (46) to (48);
            \draw [style=blue edge] (46) to (49);
            \draw [style=blue edge] (46) to (50);
            \draw [style=blue edge] (46) to (51);
            \draw [style=blue edge] (46) to (52);
            \draw [style=blue edge] (46) to (53);
            \draw [style=blue edge] (46) to (57);
            \draw [style=blue edge] (46) to (58);
            \draw [style=blue edge] (46) to (59);
            \draw [style=blue edge] (47) to (48);
            \draw [style=blue edge] (47) to (49);
            \draw [style=blue edge] (47) to (50);
            \draw [style=blue edge] (47) to (51);
            \draw [style=blue edge] (47) to (52);
            \draw [style=blue edge] (47) to (53);
            \draw [style=blue edge] (47) to (54);
            \draw [style=blue edge] (47) to (55);
            \draw [style=blue edge] (47) to (56);
            
	\end{pgfonlayer}
\end{tikzpicture}
\caption{$(S_3\square S_4)^1$ on the left, $(S_3\square S_4)^2$ in the center, $(S_3\square S_4)^3$ on the right.}
\label{power3}
\end{figure}

The bottom row seen in Figure~\ref{power3} consists of the vertices $(w_i)_{v_j}$ such that $i,j\neq 0$. There are $nm$ of these vertices.


In the third graph power, for each $i,j>0$, the vertex $(w_i)_{v_j}$ is adjacent to $(w_0)_{v_0}$, $n$ of the vertices $(w_i)_{v_0}$ such that $i\neq0$, $m$ of the vertices $(w_0)_{v_j}$ such that $j\neq0$, $m-1$ corresponding vertices $(w_i)_{v_k}$ such that $k\neq0, j$ and $i\ne 0$ is fixed, and $n-1$ vertices $(w_l)_{v_j}$ such that $l\ne 0,i$ and $j\neq0$ is fixed.
So, the degree of each of the vertices $(w_i)_{v_j}$ such that $i,j\neq0$ is $1+n+m+(m-1) + (n-1)= 2n+2m-1$.

In summary, $n+m+1$ vertices have degree $nm+n+m$, and the remaining $nm$ have degree $2n+2m-1$.
Observe that $nm>n+m-1$ for $n,m\ge2$.
Then, $nm+n+m>2n+2m-1$.
Therefore, all the vertices in the graph have at least degree $2n+2m-1$, and since $nm+n+m+1>nm+n+m>2n+2m-1$, then there exist $2n+2m$ vertices of at least degree $2n+2m-1$. The $m$-degree cannot be any higher because there are not $2n+2m+1$ vertices of any degree higher than $2n+2m-1$. Therefore the $m$-degree is $2n+2m$.

\end{proof}

The bounds for $K_n\square K_m$ have been found in \cite{qn}.
We will use the following lemma, which can be found in \cite[Proposition~7, part c]{qn}.

\begin{lemma}\label{kouider_lemma}
    \textit{Let $K_p$ be the complete graph on $p$ vertices. We show the following:}

(a) \textit{If $p\leq n<p(p-1)$, then $n\leq\varphi(K_n\square K_m)\leq m(m-1)$.}

(b) \textit{If $n\ge p(p-1)$, then $\varphi(K_n\square K_m)=n = \chi(K_n\square K_m)$.}
\end{lemma}

\begin{thm} \label{thm4.2}
    \textit{Let $G$ be the $k$-th graph power of the Cartesian product of two graphs $S_n$ and $S_m$, denoted $(S_n\square S_m)^k$, where $n\ge m$. Then the following are true of $\varphi(G)$,}

    \begin{enumerate}
        \item \textit{If $k=1$, then $\varphi(G)= m+2$}

        \item \textit{If $k=2$, then $\varphi(G) = \begin{cases}
            m+n+1, & \textit{If $n>m$}\\
            2n+2, & \textit{If $n=m$}
        \end{cases}$}

        \item \textit{If $k=3$ and $n\ge m$, then;}
        \begin{enumerate}
            \item If $m\leq n< m(m-1)$, then $2n+m+1\leq \varphi(G)\leq m^2+n+1$.
            \item If $n\ge m(m-1)$, then $\varphi(G)=2n+m+1$.
        \end{enumerate}

        \item \textit{If $k\ge4$, then $\varphi(G)= nm+n+m+1$}
    \end{enumerate}
\end{thm}

\begin{proof} We will begin with a proof by cases. Clearly, as $k$ increases, more connections are available and the $b$-chromatic number increases.

\textbf{Case I.} Let $k=1$. Observe, $(S_n\square S_m)^1=S_n\square S_m$. So, by Theorem~\ref{Theorem_box_prod}, $\varphi\left((S_n\square S_m)^1\right)=m+2$.

\textbf{Case II.} Let $k=2$, meaning that there is an edge between any two vertices with distance two. In Figure~\ref{power2} is the graph $(S_3\square S_3)^2$, the Cartesian product of star graphs $S_3$ and $S_3$ raised to the power of 2. Additionally, in Figure~\ref{power2}, we've highlighted in red the connections added from raising the graph to the power of 2.

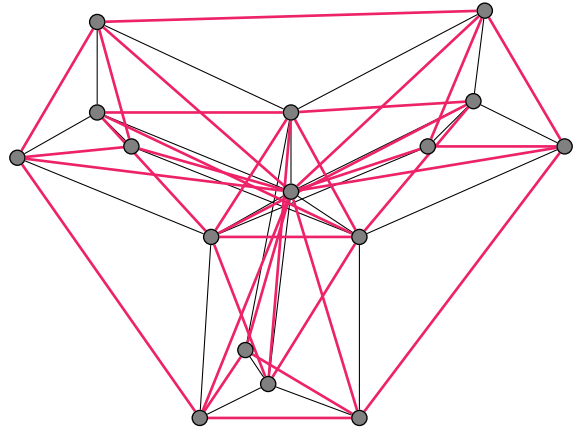
\begin{figure}[h]
    \centering
   \begin{tikzpicture}[
every edge/.style = {draw=black,very thin},
 vrtx/.style args = {#1/#2}{%
      circle, draw, thick, fill=white,
      minimum size=5mm, label=#1:#2}, scale=0.6]
	\begin{pgfonlayer}{nodelayer}
		\node [style=White]  (0) at (0, 0) {};
		\node [style=White]  (1) at (0, 1.75) {};
		\node [style=White]  (2) at (-1.75, -1) {};
		\node [style=White]  (3) at (1.5, -1) {};
		\node [style=White]  (4) at (4, 2) {};
		\node [style=White]  (5) at (-0.5, -4.25) {};
		\node [style=White]  (6) at (-4.25, 1.75) {};
		\node [style=White]  (7) at (-4.25, 3.75) {};
		\node [style=White]  (8) at (-6, 0.75) {};
		\node [style=White]  (9) at (-3.5, 1) {};
		\node [style=White]  (10) at (4.25, 4) {};
		\node [style=White]  (11) at (3, 1) {};
		\node [style=White] (12) at (6, 1) {};
		\node [style=White]  (13) at (-1, -3.5) {};
		\node [style=White]  (14) at (1.5, -5) {};
		\node [style=White]  (15) at (-2, -5) {};
	\end{pgfonlayer}
	\begin{pgfonlayer}{edgelayer}
		\draw (6) to (0);
		\draw (4) to (0);
		\draw (5) to (0);
		\draw (2) to (0);
		\draw (1) to (0);
		\draw (3) to (0);
		\draw (8) to (2);
		\draw (9) to (3);
		\draw (8) to (6);
		\draw (6) to (9);
		\draw (6) to (7);
		\draw (7) to (1);
		\draw (10) to (4);
		\draw (4) to (12);
		\draw (4) to (11);
		\draw (11) to (2);
		\draw (12) to (3);
		\draw (10) to (1);
		\draw (13) to (5);
		\draw (5) to (14);
		\draw (5) to (15);
		\draw (13) to (1);
		\draw (14) to (3);
		\draw (15) to (2);
		\draw [style=red edge] (0) to (9);
		\draw [style=red edge] (0) to (7);
		\draw [style=red edge] (0) to (8);
		\draw [style=red edge] (0) to (10);
		\draw [style=red edge] (0) to (11);
		\draw [style=red edge] (0) to (12);
		\draw [style=red edge] (0) to (13);
		\draw [style=red edge] (0) to (14);
		\draw [style=red edge] (0) to (15);
		\draw [style=red edge] (4) to (3);
		\draw [style=red edge] (3) to (6);
		\draw [style=red edge] (3) to (5);
		\draw [style=red edge] (2) to (6);
		\draw [style=red edge] (2) to (4);
		\draw [style=red edge] (2) to (5);
		\draw [style=red edge] (1) to (6);
		\draw [style=red edge] (1) to (4);
		\draw [style=red edge] (1) to (5);
		\draw [style=red edge] (10) to (12);
		\draw [style=red edge] (12) to (11);
		\draw [style=red edge] (11) to (10);
		\draw [style=red edge] (7) to (9);
		\draw [style=red edge] (7) to (8);
		\draw [style=red edge] (8) to (9);
		\draw [style=red edge] (13) to (15);
		\draw [style=red edge] (15) to (14);
		\draw [style=red edge] (14) to (13);
		\draw [style=red edge] (2) to (3);
		\draw [style=red edge] (3) to (1);
		\draw [style=red edge] (1) to (2);
		\draw [style=red edge] (12) to (14);
		\draw [style=red edge] (15) to (8);
		\draw [style=red edge] (7) to (10);
	\end{pgfonlayer}
\end{tikzpicture}
\caption{$(S_3\square S_3)^2$}
\label{power2}
\end{figure}

By Lemma \ref{lemma 5} and Proposition \ref{squeeze}, $\varphi((S_n\square S_m)^2)\leq n+m+2$.
We claim there exists a complete sub-graph $K_{m+n+1}\subsetneq (S_m\square S_n)^2$. Suppose that,

$$V_j = \left\{j\in[m]:(w_0)_{v_j}\right\} \text{ and } V_i = \left\{i\in[n]:(w_i)_{v_0}\right\}$$

Clearly, if $v\in V_j$ and $u\in V_i$, then $(w_0)_{v_0}\sim u$ and $(w_0)_{v_0}\sim v$ by the definition of the Cartesian product. Then, once we apply the graph power, transitively, $v\sim u$ and symmetrically, $u\sim v$. Therefore, each $u$, $v$, and $(w_0)_{v_0}$ form the complete sub-graph $K_{|V_j|+|V_i|+|\{(w_0)_{v_0}\}|}= K_{m+n+1}$. Therefore, $\varphi\left((S_n\square S_m)^2\right)\ge m+n+1$. 

 In any $b$-coloring, we must start by assigning each vertex in the complete sub-graph, $K_{m+n+1}$, a different color from $[1,m+n+1]$.
 Then, all that we have left to color are the vertices $(w_i)_{v_j}$ such that $i,j\neq0$.
If we want to use $m+n+2$ colors, then each vertex of the form  $(w_i)_{v_0}$ must be adjacent to the new color $m+n+2$.
But $(w_i)_{v_0}$ is only adjacent to vertices of the form $(w_i)_{v_j}$, where $j\ne 0$, and $i$ is fixed. Observe, no pair of corresponding leaves, $(w_i)_{v_j}$ and $(w_i)_{v_k}$ can have the same color.
It follows that we can only color at most one unique leaf in each star $(S_n)_{v_j}$ with the new color, $m+n+2$.
Thus, we must have an injection from the set of vertices of the form $(w_i)_{v_0}$ to the set of stars $(S_n)_{v_j}$.
Therefore, a $b$-coloring using $m+n+2$ colors is only possible if $n=m$.
If $n=m$, then we color $(w_i)_{v_i}$ the new color $m+n+2$. 

If $m<n$, then the $b$-chromatic number cannot be $m+n+2$, so $\varphi((S_n\square S_m)^2)\leq m+n+1$. Then, by the first claim, there exists a complete subgraph $K_{m+n+1}\subset(S_n\square S_m)^2$, meaning the clique number is at least $n+m+1$. Therefore, if $m<n$, then $\varphi((S_n\square S_m)^2)=m+n+1$.

\textbf{Case III.} Lemma \ref{Lemma 6} suggests that the maximum coloring uses $2n+2m$ colors. We will show tighter bounds on $\varphi$. 

The induced subgraph of $S\subset (S_n\square S_m)^3$ of all vertices, $(w_i)_{v_j}$ for $i,j\neq0$ is isomorphic to $K_n\square K_m$.
This is because any two vertices $(w_i)_{v_j}$ and $(w_l)_{v_k}$, where $i,j,l,k\neq0$, are adjacent in the third graph power if and only if  $i=l$, or $j=k$.
Each vertex in $(S_n\square S_m)^3$ outside of the subgraph $S$ is adjacent to all of the other vertices in the third graph power.


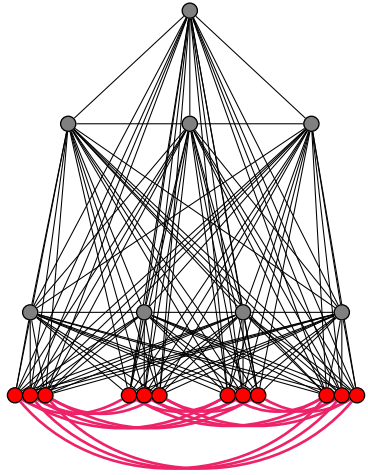
\begin{figure}[h]
\centering

\begin{tikzpicture}[scale=0.4]
	\begin{pgfonlayer}{nodelayer}
		\node [style=White] (40) at (24.5, 12.75) {};
		\node [style=White] (41) at (20.5, 9) {};
		\node [style=White] (42) at (24.5, 9) {};
		\node [style=White] (43) at (28.5, 9) {};
		\node [style=White] (44) at (19.25, 2.75) {};
		\node [style=White] (45) at (23, 2.75) {};
		\node [style=White] (46) at (26.25, 2.75) {};
		\node [style=White] (47) at (29.5, 2.75) {};
		\node [style=Red] (48) at (18.75, 0) {};
		\node [style=Red] (49) at (19.25, 0) {};
		\node [style=Red] (50) at (19.75, 0) {};
		\node [style=Red] (51) at (22.5, 0) {};
		\node [style=Red] (52) at (23, 0) {};
		\node [style=Red] (53) at (23.5, 0) {};
		\node [style=Red] (54) at (25.75, 0) {};
		\node [style=Red] (55) at (26.25, 0) {};
		\node [style=Red] (56) at (26.75, 0) {};
		\node [style=Red] (57) at (29, 0) {};
		\node [style=Red] (58) at (29.5, 0) {};
		\node [style=Red] (59) at (30, 0) {};
	\end{pgfonlayer}
	\begin{pgfonlayer}{edgelayer}
            \draw (45) to (46);
            \draw (46) to (47);
            \draw (44) to (45);
            \draw (42) to (43);
            \draw (41) to (42);
            \draw [bend right, style=red edge] (48) to (51);
            \draw [bend right, style=red edge] (48) to (54);
            \draw [bend right=50, style=red edge] (48) to (57);
            \draw [bend right, style=red edge] (49) to (52);
            \draw [bend right, style=red edge] (49) to (55);
            \draw [bend right=50, style=red edge] (49) to (58);
            \draw [bend right, style=red edge] (50) to (53);
            \draw [bend right, style=red edge] (50) to (56);
            \draw [bend right=50, style=red edge] (50) to (59);
            \draw [bend right, style=red edge] (51) to (54);
            \draw [bend right, style=red edge] (52) to (55);
            \draw [bend right, style=red edge] (53) to (56);
            \draw [bend right, style=red edge] (51) to (57);
            \draw [bend right, style=red edge] (52) to (58);
            \draw [bend right, style=red edge] (53) to (59);
            \draw [bend right, style=red edge] (54) to (57);
            \draw [bend right, style=red edge] (55) to (58);
            \draw [bend right, style=red edge] (56) to (59);
		\draw (40) to (41);
		\draw (40) to (42);
		\draw (40) to (43);
		\draw (40) to (44);
		\draw (40) to (45);
		\draw (40) to (46);
		\draw (40) to (47);
		\draw (44) to (48);
		\draw (44) to (49);
		\draw (44) to (50);
		\draw (45) to (51);
		\draw (45) to (52);
		\draw (45) to (53);
		\draw (46) to (54);
		\draw (46) to (55);
		\draw (46) to (56);
		\draw (47) to (57);
		\draw (47) to (58);
		\draw (47) to (59);
		\draw (41) to (48);
		\draw (41) to (51);
		\draw (41) to (54);
		\draw (41) to (57);
		\draw (42) to (49);
		\draw (43) to (50);
		\draw (42) to (52);
		\draw (43) to (53);
		\draw (42) to (55);
		\draw (43) to (56);
		\draw (42) to (58);
		\draw (43) to (59);
		\draw (48) to (49);
		\draw (49) to (50);
		\draw (51) to (52);
		\draw (52) to (53);
		\draw (54) to (55);
		\draw  (55) to (56);
		\draw  (57) to (58);
		\draw  (58) to (59);
		\draw (41) to (44);
		\draw (41) to (45);
		\draw  (41) to (46);
		\draw (41) to (47);
		\draw (42) to (44);
		\draw (42) to (45);
		\draw  (42) to (46);
		\draw (42) to (47);
		\draw (43) to (44);
		\draw (43) to (45);
		\draw (43) to (46);
		\draw  (43) to (47);
		\draw  (40) to (48);
		\draw (40) to (49);
		\draw (40) to (50);
		\draw  (40) to (51);
		\draw  (40) to (52);
		\draw (40) to (53);
		\draw  (40) to (54);
		\draw  (40) to (55);
		\draw  (40) to (56);
		\draw  (40) to (57);
		\draw (40) to (58);
		\draw  (40) to (59);
		\draw (41) to (50);
		\draw  (41) to (49);
		\draw  (41) to (52);
		\draw  (41) to (53);
		\draw  (41) to (55);
		\draw  (41) to (56);
		\draw  (41) to (58);
		\draw (41) to (59);
		\draw  (42) to (48);
		\draw  (42) to (50);
		\draw  (42) to (51);
		\draw  (42) to (53);
		\draw  (42) to (54);
		\draw   (42) to (56);
		\draw   (42) to (57);
		\draw   (42) to (59);
		\draw   (43) to (49);
		\draw   (43) to (48);
		\draw   (43) to (51);
		\draw   (43) to (52);
		\draw   (43) to (54);
		\draw   (43) to (55);
		\draw   (43) to (57);
		\draw   (43) to (58);
            \draw   (44) to (51);
            \draw   (44) to (52);
            \draw   (44) to (53);
            \draw   (44) to (54);
            \draw   (44) to (55);
            \draw   (44) to (56);
            \draw   (44) to (57);
            \draw   (44) to (58);
            \draw   (45) to (48);
            \draw   (45) to (49);
            \draw   (45) to (50);
            \draw   (45) to (54);
            \draw   (45) to (55);
            \draw   (45) to (56);
            \draw   (45) to (57);
            \draw   (45) to (58);
            \draw   (45) to (59);
            \draw   (46) to (48);
            \draw   (46) to (49);
            \draw   (46) to (50);
            \draw   (46) to (51);
            \draw   (46) to (52);
            \draw   (46) to (53);
            \draw   (46) to (57);
            \draw   (46) to (58);
            \draw   (46) to (59);
            \draw   (47) to (48);
            \draw   (47) to (49);
            \draw   (47) to (50);
            \draw   (47) to (51);
            \draw   (47) to (52);
            \draw   (47) to (53);
            \draw   (47) to (54);
            \draw   (47) to (55);
            \draw   (47) to (56);
            
	\end{pgfonlayer}
\end{tikzpicture}
\caption{The graph $(S_3\square S_4)^3$ with the sub-graph $K_3\square K_4$ highlighted in red.}
\label{power3,2}
\end{figure}

Observe, we must have at least $m+n+1$ unique colors to color the $m+n+1$ vertices which are outside of $S$. These colors will not be repeated in the sub-graph $S\cong K_n\square K_m$, as that would violate the definition of a proper coloring. Our new question simplifies to, how can we maximize the coloring of $K_n\square K_m$?

Lemma~\ref{kouider_lemma} shows us the bounds for $\varphi(K_n\square K_m)$. So, we can add the additional $n+m+1$ unique colors and then, $n+m+1 + n\leq \varphi((S_n\square S_m)^3)\leq n+m+1 + m(m-1)$ if $n\leq m(m-1)$. In other words, if $m\leq n< m(m-1)$, $$\chi((S_n\square S_m)^3)\leq \varphi((S_n\square S_m)^3)\leq m^2+n+1.$$ 
By part (b) of Lemmma~\ref{kouider_lemma}, if $n\ge m(m-1)$, then $\varphi((S_n\square S_m)^3)=\chi((S_n\square S_m)^3)$. This concludes the proof for Case III.

\textbf{Case IV.} Let $k\ge 4$. 

$Diam\left((S_n\square S_m)^4\right)=4$. By Fact~\ref{fact 1}, for $k\ge 4,\varphi\left(\left(S_n\square S_m\right)^k\right)=\left|V\left(\left(S_n\square S_m\right)^k\right)\right|=\left|V(S_n)\times V(S_m)\right|=(n+1)(m+1)=mn+n+m+1$.
\end{proof}

From Theorem \ref{thm4.2}, we can pull some direct corollaries and exact values for the $b$-chromatic number of particular cases of $K_n\square K_m$.

\begin{thm}\label{rookgraphcor}
   \textit{It follows that,} $\varphi((S_1\square S_3)^3) = 8$, $\varphi((S_2\square S_3)^3) = 9$, $\varphi((S_3\square S_3)^3) = 10$, $\varphi((S_4\square S_3)^3) = 13$, $\varphi((S_5\square S_3)^3) = 15$. For $n\ge 6$, $\varphi((S_n\square S_3)^3) = 2n+4$.
\end{thm}
\begin{proof}
    For the graph $G=(S_n\square S_m)^3$, we fix $m=3$.
    Recall the relation between $(S_n\square S_m)^3$ and $K_n\square K_m$. We show the maximum $b$-coloring of $K_n\square K_m$ first, and then we add $n+m+1$ unique colors for the remaining vertices of $S_n\square S_m$.
    In other words, it suffices to show,

    \begin{center}
        $\begin{array}{c}
             \varphi(K_1\square K_3)=\varphi(K_2\square K_3)=\varphi(K_3\square K_3)=3  \\
              \varphi(K_4\square K_3)=5\\
              \varphi(K_5\square K_3)=6\\
        \end{array}$
    \end{center}
    
     Then, we'll show for $n>5$, $\varphi(K_n\square K_3)=n$. 
     
    First, observe $K_1\square K_3\cong K_3$ which has a maximum $b$-coloring of 3. Then, $3+1+1 + 3 = 8$. 
    Similarly, we check $K_2\square K_3\cong Y_3$ (the prism graph), which clearly cannot include a 4th $b$-vertex. This was extensively studied in \cite{Ansari}. Then, $3+2+1 + 3 = 9$. 
    For the last three cases, observe that $K_n\square K_3$ can't have a $b$-coloring containing $n+2$ vertices. Suppose, for the sake of contradiction, that a graph coloring does contain $n+2$ colors. Then one copy of $K_n$, $(K_n)_{v_1}$ in particular, must contain unique colors $\{1,2,\ldots, n\}$ and by symmetry let the vertices with colors 1 and 2 be b-vertices.
    Let $(w_1)_{v_1}$ be the vertex which is colored 1, and let  $(w_2)_{v_1}$ be the vertex colored 2.
    If these vertices are $b$-vertices, then they must be adjacent to colors $n+1$ and $n+2$.
    The only available places to color vertices adjacent to 1 and 2 are in the corresponding vertices in the other inner graphs $(K_n)_{v_2}$ and $(K_n)_{v_3}$. 
    By symmetry color $(w_1)_{v_2}$ with the color $n+1$ and $(w_2)_{v_2}$ with the color $n+2$.
    In $(K_n)_{v_3}$, we color $(w_1)_{v_3}$ with color $n+2$, and we color $(w_2)_{v_3}$ with color $n+1$ (so both $(w_1)_{v_1}$ and $(w_1)_{v_2}$ are b-vertices and the coloring is proper).
    Then, if this is a $b$-coloring,  one of our vertices with the color $n+2$ must be a $b$-vertex.
    By symmetry, suppose it is $(w_2)_{v_2}$.
    Observe that this vertex is adjacent to color $n+1$ and color 2, and still needs to be adjacent to the $(n-1)$ remaining colors.
    But there are only $(n-2)$ uncolored vertices which are adjacent to $(w_2)_{v_2}$.
 Therefore, $\varphi(K_n\square K_3)\leq n+1$ for $n\in[5]$.
    Then, we have the following $b$-colorings for $\varphi(K_3\square K_3)=3$, $\varphi(K_4\square K_3)=5$, and $\varphi(K_5,\square K_3)=6$ in Figure \ref{Figure 10}, Figure \ref{Figure 11}, and Figure \ref{Figure 12} respectively. We write each inner graph $K_n$ as a column and each outer graph $K_m$ as a row. The circled numbers are our $b$-vertices. Note that all colors in a row must be distinct, and all colors in a column must also be distinct. In Figure~\ref{gridmap} is the mapping from the collection of vertices $\bigcup_{j=1}^n\bigcup_{i=1}^n(w_i)_{v_j}=V(K_m\square K_n)$ to an $n\times n$ grid for $n=3$.

    \begin{figure}[h]
        \centering
        \begin{tabular}{|c|c|c|}
        \hline
             $(w_1)_{v_1}$ & $(w_1)_{v_2}$ & $(w_1)_{v_3}$ \\
             \hline
             $(w_2)_{v_1}$ & $(w_2)_{v_2}$ & $(w_2)_{v_3}$\\
             \hline
              $(w_3)_{v_1}$& $(w_3)_{v_2}$ &  $(w_3)_{v_3}$\\
             \hline
        \end{tabular}
        \caption{Grid for $K_3\square K_3$}
        \label{gridmap}
    \end{figure}

    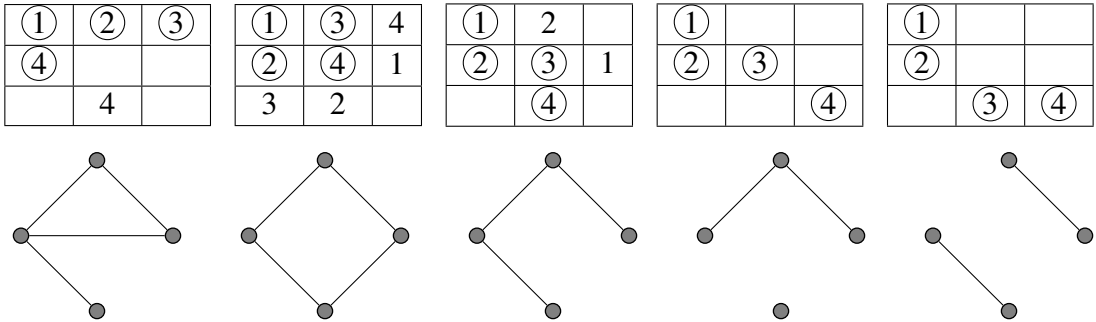
\begin{figure}[h]
        \centering
        \begin{tabular}{|c|c|c|}
        \hline
             \circled{1} & \circled{2} & \circled{3} \\
             \hline
             \circled{4} &  & \\
             \hline
              & 4 &  \\
             \hline
        \end{tabular} \hspace{0.1cm}
        \begin{tabular}{|c|c|c|}
        \hline
             \circled{1} & \circled{3} & 4 \\
             \hline
             \circled{2} & \circled{4} &1 \\
             \hline
             3 & 2 &  \\
             \hline
        \end{tabular} \hspace{0.1cm}
        \begin{tabular}{|c|c|c|}
        \hline
             \circled{1} & 2 &  \\
             \hline
             \circled{2} & \circled{3} & 1\\
             \hline
              & \circled{4} &  \\
             \hline
        \end{tabular}
        \hspace{0.1cm}
        \begin{tabular}{|c|c|c|}
        \hline
             \circled{1} &  &  \\
             \hline
             \circled{2} & \circled{3} & \\
             \hline
              &  & \circled{4} \\
             \hline
        \end{tabular}  \hspace{0.1cm}
        \begin{tabular}{|c|c|c|}
        \hline
             \circled{1} &  &  \\
             \hline
             \circled{2} &  & \\
             \hline
              & \circled{3} & \circled{4} \\
             \hline
        \end{tabular}
        \vspace{0.3cm}

        \begin{tikzpicture}
	\begin{pgfonlayer}{nodelayer}
		\node [style=White] (0) at (0, 1) {};
		\node [style=White] (1) at (-1, 0) {};
		\node [style=White] (2) at (1, 0) {};
		\node [style=White] (3) at (0, -1) {};
		\node [style=White] (4) at (2, 0) {};
		\node [style=White] (5) at (3, 1) {};
		\node [style=White] (6) at (4, 0) {};
		\node [style=White] (7) at (3, -1) {};
		\node [style=White] (8) at (5, 0) {};
		\node [style=White] (9) at (6, 1) {};
		\node [style=White] (10) at (7, 0) {};
		\node [style=White] (11) at (6, -1) {};
		\node [style=White] (12) at (8, 0) {};
		\node [style=White] (13) at (9, 1) {};
		\node [style=White] (14) at (10, 0) {};
		\node [style=White] (15) at (9, -1) {};
		\node [style=White] (16) at (11, 0) {};
		\node [style=White] (17) at (12, 1) {};
		\node [style=White] (18) at (13, 0) {};
		\node [style=White] (19) at (12, -1) {};
	\end{pgfonlayer}
	\begin{pgfonlayer}{edgelayer}
		\draw (0) to (1);
		\draw (1) to (3);
		\draw (0) to (2);
		\draw (2) to (1);
		\draw (4) to (5);
		\draw (5) to (6);
		\draw (6) to (7);
		\draw (4) to (7);
		\draw (11) to (8);
		\draw (8) to (9);
		\draw (9) to (10);
		\draw (12) to (13);
		\draw (13) to (14);
		\draw (16) to (19);
		\draw (17) to (18);
	\end{pgfonlayer}
\end{tikzpicture}
    
        \caption{The cases for contradictions for $\varphi(K_3\square K_3)=4$.
        The cases are listed as Case 1 ($K_3\cup K_1$), Case 2 ($C_4$), Case 3 ($P_4$), Case 4 ($P_3\sqcup K_1$), and Case 5 ($C_2\sqcup C_2$).}
        \label{Figure 10}
    \end{figure}

In Figure~\ref{Figure 10}, we describe the different cases of contradictions we find for a $b$-coloring using 4 colors. What we are interested in is the induced sub-graph of the chosen $b$-vertices. Let's call this induced subgraph $G_B$. Because we want 4 b-vertices, $|G_B|=4$. We can also say, $\omega(K_3\square K_3)=\max\{\omega(K_3),\omega(K_3)\}=\omega(K_3)=3$ and so it follows that $\omega(G_B)\leq 3$. Let the inner graph $K_3$ be the vertices $u_1,u_2,u_3$ and the outer graph be the vertices $v_1,v_2,v_3$ and fix $\omega(G_B)=3$. Without loss of generality, let the clique of 3 b-vertices lie in $(K_3)_{v_1}$ (or lie in the outer graph $\{(u_1)_{v_k}:k\in[3]\}$). The fourth vertex thus cannot be in $(K_3)_{v_1}$, so suppose $(u_k)_{v_j}$ for $k\in[3]$, $j\neq 1$, has the fourth color. This vertex is forced to be adjacent to only $(u_k)_{v_1}$. Thus, if $\omega(G_B)=3$, the degree of our fourth $b$-vertex \textit{must} be 1. This rules out the sub-cases where we have a clique number of 3, and so, the degree of our fourth vertex is 0,2, or 3.

For this case, shown in Case 1 in Figure~\ref{Figure 10}, the contradiction starts with assigning vertices adjacent to the $b$-vertices the fourth color. We see that if we have a complete graph on three vertices, then they must all be in the same row or all in the same column, thus, the number four must be listed 3 distinct times to be adjacent to all the $b$-vertices. But that is not possible as we can only have 2 remaining copies of the inner graph to place our fourth color, as shown in Case 1 of Figure~\ref{Figure 10}.

Then, we claim that $G_B$ has at most 2 components. We'll show that it is impossible to construct an induced subgraph $G_B$ with 3 components. It will follow that we can't construct a subgraph with 4 components, because this would contain a subgraph of 3 components. Let each $(u_n)_{v_n}$ for $n\in[3]$ be a $b$-vertex. These are distinct and disconnected. Without loss of generality, choose a fourth $b$-vertex arbitrarily in $(K_3)_{v_1}$. This vertex, say for example, $(u_2)_{v_1}$, is by definition adjacent to $(u_1)_{v_1}$ and $(u_2)_{v_2}$, and thus has a degree of 2. So, a construction of an induced subgraph with 3 disconnected parts is not possible, and thus, we only need to consider cases with at most 2 disconnected parts. So, we consider induced subgraphs $G_B$ with $\omega(G_B)= 2$ and with at most 2 connected components. These subgraphs are considered in Cases 2--5.

Suppose that we have an induced subgraph of $b$-vertices, $C_4$, as shown in Case 2 of Figure~\ref{Figure 10}. Clearly, if all the colors adjacent to our $b$-vertices are distinct, then the contradiction is immediate, as we cannot properly color the last vertex. So we will show that these colors \textit{must} be distinct. For the sake of contradiction, suppose that the colors are not distinct. That is, for colors $a,b,c,d\in C$, one pair of colors is repeated, or without loss of generality, $a=d$. Consider the following grid coloring.
\begin{center}
\begin{tabular}{|c|c|c|}
        \hline
             \circled{1} & \circled{3} & a \\
             \hline
             \circled{2} & \circled{4} &b \\
             \hline
             d & c &  \\
             \hline
        \end{tabular}
        \quad
        \begin{tabular}{|c|c|c|}
        \hline
             \circled{1} & \circled{3} & 4 \\
             \hline
             \circled{2} & \circled{4} &3 \\
             \hline
             4 & 2 &  \\
             \hline
        \end{tabular}
\end{center}

Since $a$ and $d$ are the only vertices adjacent to the $b$-vertex 1, either one of these must be the color 4. Since $a=d$, we must color them both 4. This forces $b=3$ and $c=2$ for the vertices labeled 2 and 3 to be $b$-vertices (as shown in the right-hand table above). However, with this coloring, the vertex labeled 4 can't be a $b$-vertex, because it is not adjacent to color 1. 
Thus, we have a contradiction. Therefore, $a\neq b\neq c\neq d$ and we have our initial contradiction.

Next, we consider Case 3 (shown third from the left in Figure~\ref{Figure 10}).
In this case, $(w_1)_{v_1}$ is the $b$-vertex with color $1$, $(w_2)_{v_1}$ is the $b$-vertex with color 2, $(w_2)_{v_2}$ is the $b$-vertex with color 3, and $(w_3)_{v_2}$ is the $b$-vertex with color 4.
Observe $(w_1)_{v_2}$ must have color 2, because it is already adjacent to vertices with colors 1,3, and 4.
If $(w_2)_{v_2}$ is a $b$-vertex, it must be adjacent to a vertex with color 1.
The only vertex adjacent to $(w_2)_{v_2}$ which has not already been colored is $(w_2)_{v_3}$.
Thus, we must assign $(w_2)_{v_3}$ the color 1.
The $b$-vertex labeled 4 must be adjacent to the first color, and we cannot fill any of the spaces left with the first color; we have a contradiction.

For Case 4, the vertex $(w_2)_{v_1}$ is a $b$-vertex and thus must be adjacent to color 4. The only vertices that are colorable in the neighborhood of $(w_2)_{v_1}$ are $(w_3)_{v_1}$ and $(w_2)_{v_3}$. These are both adjacent to $(w_3)_{v_3}$, which already has the color 4. So, we can't have a $b$-coloring with 4 colors.

Lastly, in Case 5, our subgraph cannot be properly colored with 4 colors, as the bottom left vertex, $(w_3)_{v_1}$, is already adjacent to all four colors.

Thus, no graph $K_3\square K_3$ can have four $b$-vertices. Thus, $\varphi(K_3\square K_3)=3$ and we obtain $\varphi((S_3\square S_3)^3)=(m+n+1)+3=(3+3+1)+3=10$.

 Hence, if $\varphi(K_3\square K_3)<4$, then $\varphi(K_3\square k_3)\leq 3$. By Proposition \ref{prop2.2}, $\varphi(K_3\square K_3)\ge \max\{\varphi(K_3),\varphi(K_3)\}=\max\{3,3\}=3$. Therefore, $\varphi(K_3\square K_3)=3$.

\begin{figure}[h]
    \centering
    \begin{tabular}{|c|c|c|}
        \hline
            \circled{1} & 5 & 4\\
        \hline
            \circled{2} & 3 & \circled{5}\\
        \hline
            3 & \circled{4} & 1\\
        \hline
            4& 2 & \circled{3} \\
        \hline
    \end{tabular}
    \caption{A $b$-coloring of $K_4\square K_3$ using 5 colors.}
    \label{Figure 11}
\end{figure}

 By Figure~\ref{Figure 11}, $\varphi(K_4\square K_3)=5$ and we obtain $\varphi((S_4\square S_3)^3)=(m+n+1)+3=(4+3+1)+5=13$.

    \begin{figure}[h]
        \centering
        \begin{tabular}{|c|c|c|}
        \hline
             \circled{1} & 6 & 4 \\
             \hline
             \circled{2} & \circled{3} & \circled{6} \\
             \hline
             3 & 1 & \circled{5}  \\
             \hline
             4 & 5 & 1 \\
             \hline
             5 & \circled{4} & 2 \\
             \hline
        \end{tabular}
        
        \caption{A $b$-coloring of $K_5\square K_3$ using 6 colors.}
        \label{Figure 12}
    \end{figure}

     By Figure~\ref{Figure 12}, $\varphi(K_5\square K_3)=6$ and we obtain $\varphi((S_5\square S_3)^3)=(m+n+1)+6=(5+3+1)+6=15$.

     If $n>5$, then observe that $n\ge 3(3-1)=6$, so by Lemma~\ref{kouider_lemma}, $\varphi((S_n\square S_3)^3) = 2n+4$. 
\end{proof}

\begin{rem}\label{rem4.1}
    Using a Python program, we generated 12 b-colorings for $K_3\square K_3$ using 3 colors in 0.029 seconds, 11,384 $b$-colorings for $K_4\square K_3$ using 5 colors in 1.565 seconds, and 570,240 $b$-colorings with 6 colors in 3.48 minutes. Note that many of these $b$-colorings are isomorphic to one another, simply by switching rows or switching columns. It was shown in \cite{Manlove} that finding the $b$-chromatic number of a given graph was NP-Hard and that it was possible to find the $b$-chromatic number of a tree in polynomial time. In \cite{Maffray}, computer searches were ran for $K_5\square K_5$, $K_6\square K_6$, and $K_7\square K_7$; however, the time complexity was not mentioned. The time complexity to find the $b$-chromatic number of any complete graph is $\mathcal{O}(2^n)$ for $K_n\square K_n$. It's also interesting to note that the probability of $K_3\square K_3$ colored to be a $b$-coloring using 3 colors is 0.06\%, the probability of $K_4\square K_3$ colored to be a $b$-coloring using 5 colors is 0.004\%, and the probability of $K_5\square K_3$ colored to be a $b$-coloring using 6 colors is $0.0001\%$.

        For larger $m$, the number of cases grows at the same rate as the number of integer partitions of $n$, making larger examples difficult to compute by hand. For each case, there are sub-cases for different choices made while coloring.
    \end{rem}

\section{Open Problems}

\begin{prob}
    In Remark \ref{rem4.1}, we mentioned the probabilities of a colored graph to be a $b$-coloring. Clearly, $P(G\text{ being a $b$-coloring})=\frac{B(G,k)}{k^{|V(G)|}}$, where $B(G,k)=$\textit{ the number of b-colorings of G with k colors}. Is $B(G,k)$ generalizable? Does there exist a graph where the probability of $b$-coloring existing as $k$ increases does \textit{not} tend to zero?
\end{prob}

\begin{conj}
    \textit{Let $G$ be the \textit{tensor product} of two distinct star graphs $S_n$ and $S_m$. Then, the $b$-chromatic number of $G$ is $m+n+1$.}
\end{conj}

\begin{conj}
    \textit{For all graphs $G$ and $H$ and all $k\in\N$,}
    \begin{equation*}
        \varphi((G\square H)^k)\ge\varphi(G^k\square H^k)
    \end{equation*}
\end{conj}

\begin{prob}
    
Conjecture 1 has not been verified; however, it is strongly believed to be the case. Show that $m+n+1$ is the upper bound for the $b$-chromatic number of $G$. Other products of interest are the lexicographic product of $S_n$ and $S_m$, the strong product of $S_n$ and $S_m$, and the rooted product of $S_n$ and $S_m$. \cite{Jakovac}\cite{Koch}\cite{tensor}

\end{prob}

\begin{prob}What is the relation between the total graph function and the Cartesian product of a graph onto itself? Is the $b$-chromatic number of the Cartesian product of $G\square G$ always greater than or equal to $T(G)$? It is clear that $T(G\square G)\not\cong T(G)\square T(G)$, but is there any other relation we can make?  \cite{sunlet}\cite{Vernold}
\end{prob}

\begin{prob}The power graph of a cycle $C_4$ and the power graph of a star $S_3$ are isomorphic. This becomes problematic when trying to take the inverse of the power graph because the function is not one-to-one. If a graph does not contain the sub-graphs $C_n:n\ge3$  and $S_m:m\ge3$, is there a way to define a "\textit{root graph}" function that, when given $G^p,$ will return $G$? In other words, is it possible to construct a function $\sqrt[\leftroot{-2}\uproot{2}p]{G}$? Intuition suggests this is impossible, but proving this is difficult.\end{prob}

\begin{prob}A star graph is isomorphic to the complete bipartite graph $K_{1,n}$. If we consider $\varphi(K_{1,x}\square K_{1,y}) = y+2$ for $x\ge y$ as a base case, is it possible to find $\varphi(K_{u,x}\square K_{v,y})$ by double induction?\end{prob}

\begin{prob}General bounds for the $b$-chromatic number of the Cartesian product of graphs with girths $\geq 5$ have been considered; however, bounds for girths $\le 4$ have not been found and are considered NP-hard as the girth gets too small to consider all possibilities.
\end{prob}

\begin{prob}
    An interesting problem is to characterize all graphs $G$ such that $\varphi(G)=m(G)$.
\end{prob}

\begin{prob}
    Similar methods, as seen in Theorem~\ref{rookgraphcor}, could be used to find the $b$-chromatic number of $K_n\square K_4$. How does this process generalize?
\end{prob}

\section{Acknowledgments}

This research project was funded by the DePaul University USRP. Special thanks to the students from the MAA NREU program for their contribution to our work. 
This paper was written while E. Dahlen was an undergraduate student at DePaul University. E. Dahlen would like to thank Dr. Emily Barnard, Dr. Sarah Bockting-Conrad, and Dr. Christopher Drupieski for offering their many valuable ideas and suggestions.


\bibliographystyle{amsalpha}
\bibliography{bib}

\newpage
\appendix
\section{Total Graphs Appendix}

Shown below is an example of the steps taken in the algorithm for case 1 in Theorem~\ref{thm:line_tot} to achieve a maximal $b$-coloring. This example is taken for $n=5$ and $m=3$, satisfying $n\ge 2(m-1)$.

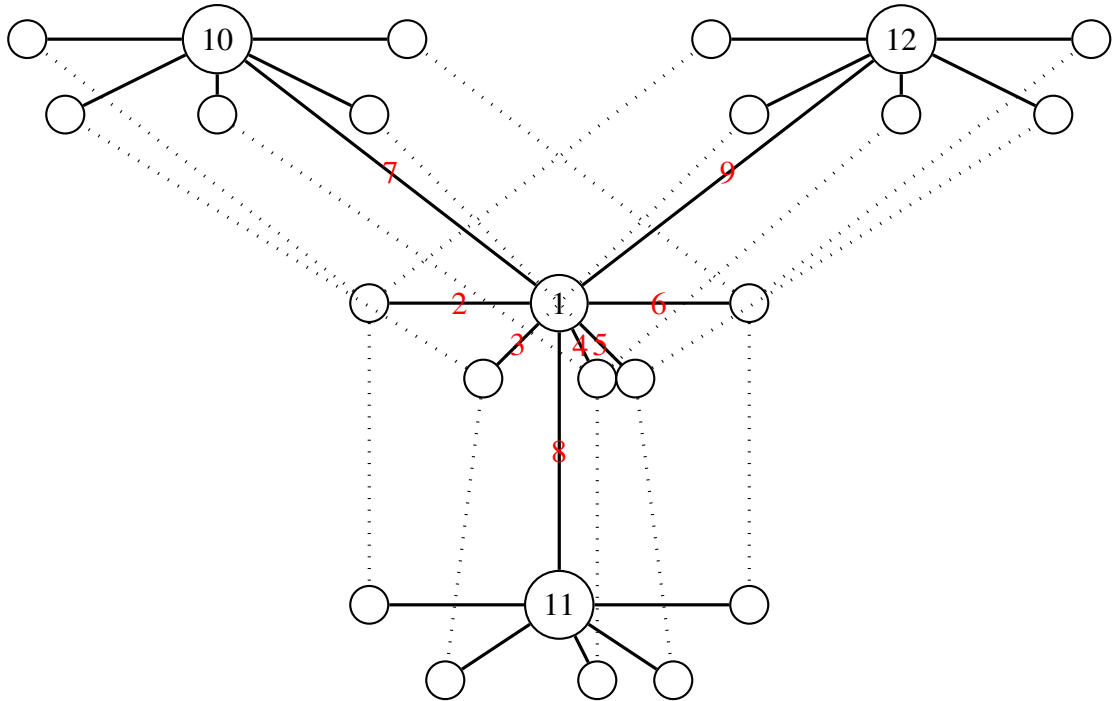
\begin{tikzpicture}[
every edge/.style = {draw=black,very thick},
 vrtx/.style args = {#1/#2}{%
      circle, draw, thick, fill=white,
      minimum size=5mm, label=#1:#2}, scale=0.5
                    ]
		\node   (0) [vrtx=left/] at (0, 8) {1};
		\node   (1) [vrtx=left/] at (-2, 6) {};
		\node   (2) [vrtx=left/] at (1, 6) {};
		\node   (3) [vrtx=left/] at (2, 6) {};
		\node   (4) [vrtx=left/] at (-5, 8) {};
		\node   (5) [vrtx=left/] at (5, 8) {};
		\node   (6) [vrtx=left/] at (-9, 15) {10};
		\node   (7) [vrtx=left/] at (-13, 13) {};
		\node   (8) [vrtx=left/] at (-9, 13) {};
		\node   (9) [vrtx=left/] at (-5, 13) {};
		\node   (10) [vrtx=left/] at (-14, 15) {};
		\node   (11) [vrtx=left/] at (-4, 15) {};
		\node   (12) [vrtx=left/] at (0, 0) {11};
		\node   (13) [vrtx=left/] at (-3, -2) {};
		\node   (14) [vrtx=left/] at (1, -2) {};
		\node   (15) [vrtx=left/] at (3, -2) {};
		\node   (16) [vrtx=left/] at (-5, 0) {};
		\node   (17) [vrtx=left/] at (5, 0) {};
		\node   (18) [vrtx=left/] at (9, 15) {12};
		\node   (19) [vrtx=left/] at (5, 13) {};
		\node   (20) [vrtx=left/] at (9, 13) {};
		\node   (21) [vrtx=left/] at (13, 13) {};
		\node   (22) [vrtx=left/] at (4, 15) {};
		\node   (23) [vrtx=left/] at (14, 15) {};

        \path
        
		  (0) edge node[midway]{\textcolor{red}{3}}(1)
		  (0) edge node[midway]{\textcolor{red}{4}}(2)
		  (0) edge node[midway]{\textcolor{red}{5}}(3)
		(4) edge node[midway]{\textcolor{red}{2}}(0)
		  (0) edge node[midway]{\textcolor{red}{6}}(5)
		  (6) edge (7)
		  (6) edge (8)
		  (6) edge (9)
		(10) edge (6)
		  (6) edge (11)
		  (12) edge (13)
		  (12) edge (14)
		  (12) edge (15)
		  (16) edge (12)
		  (12) edge (17)
		  (18) edge (19)
	    (18) edge (20)
		(18) edge (21)
	    (22) edge (18)
	    (18) edge (23)
	    (0) edge node[midway]{\textcolor{red}{7}}(6)
	    (0) edge node[midway]{\textcolor{red}{8}}(12)
            (0) edge node[midway]{\textcolor{red}{9}}(18);
\path[loosely dotted] (4) edge (10)
		  (4) edge (16)
		(4) edge (22)
		  (1) edge (7)
		  (1) edge (13)
		  (1) edge (19)
		(2) edge (8)
		  (2) edge (14)
		(2) edge (20)
		(3) edge (9)
		  (3) edge (15)
		  (3) edge (21)
		(5) edge (11)
		(5) edge (17)
		(5) edge (23);
\end{tikzpicture}

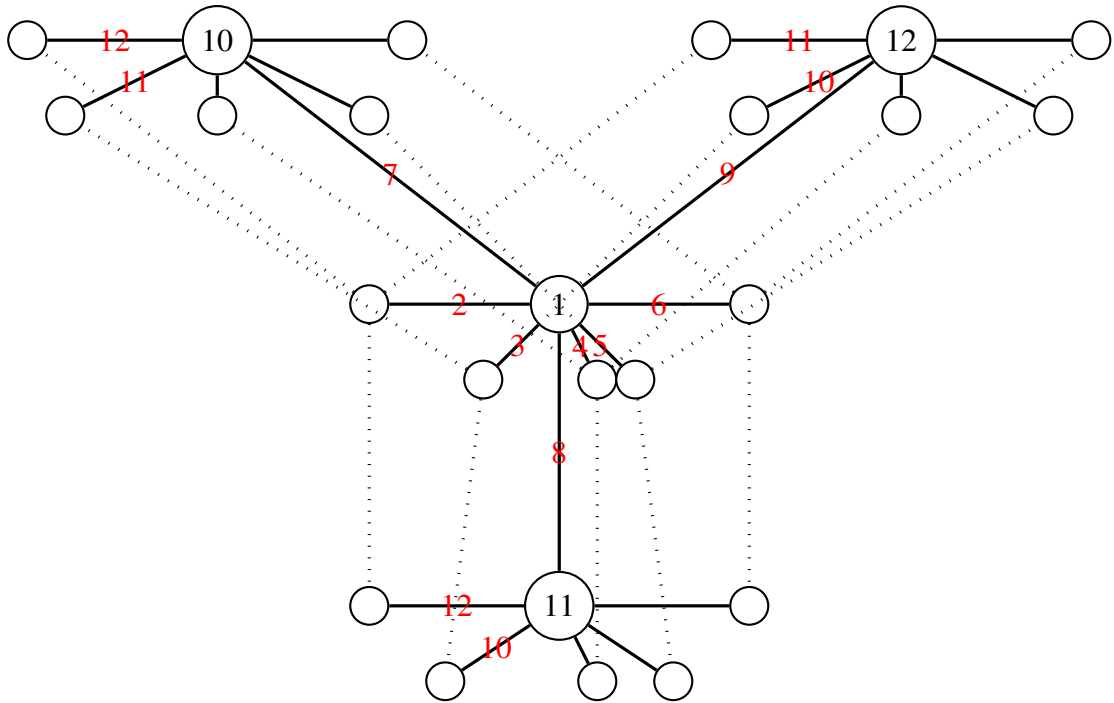
\begin{tikzpicture}[
every edge/.style = {draw=black,very thick},
 vrtx/.style args = {#1/#2}{%
      circle, draw, thick, fill=white,
      minimum size=5mm, label=#1:#2}, scale=0.5
                    ]
		\node   (0) [vrtx=left/] at (0, 8) {1};
		\node   (1) [vrtx=left/] at (-2, 6) {};
		\node   (2) [vrtx=left/] at (1, 6) {};
		\node   (3) [vrtx=left/] at (2, 6) {};
		\node   (4) [vrtx=left/] at (-5, 8) {};
		\node   (5) [vrtx=left/] at (5, 8) {};
		\node   (6) [vrtx=left/] at (-9, 15) {10};
		\node   (7) [vrtx=left/] at (-13, 13) {};
		\node   (8) [vrtx=left/] at (-9, 13) {};
		\node   (9) [vrtx=left/] at (-5, 13) {};
		\node   (10) [vrtx=left/] at (-14, 15) {};
		\node   (11) [vrtx=left/] at (-4, 15) {};
		\node   (12) [vrtx=left/] at (0, 0) {11};
		\node   (13) [vrtx=left/] at (-3, -2) {};
		\node   (14) [vrtx=left/] at (1, -2) {};
		\node   (15) [vrtx=left/] at (3, -2) {};
		\node   (16) [vrtx=left/] at (-5, 0) {};
		\node   (17) [vrtx=left/] at (5, 0) {};
		\node   (18) [vrtx=left/] at (9, 15) {12};
		\node   (19) [vrtx=left/] at (5, 13) {};
		\node   (20) [vrtx=left/] at (9, 13) {};
		\node   (21) [vrtx=left/] at (13, 13) {};
		\node   (22) [vrtx=left/] at (4, 15) {};
		\node   (23) [vrtx=left/] at (14, 15) {};

        \path
        
		  (0) edge node[midway]{\textcolor{red}{3}}(1)
		  (0) edge node[midway]{\textcolor{red}{4}}(2)
		  (0) edge node[midway]{\textcolor{red}{5}}(3)
		(4) edge node[midway]{\textcolor{red}{2}}(0)
		  (0) edge node[midway]{\textcolor{red}{6}}(5)
		  (6) edge node[midway]{\textcolor{red}{11}}(7)
		  (6) edge (8)
		  (6) edge (9)
		(10) edge node[midway]{\textcolor{red}{12}}(6)
		  (6) edge (11)
		  (12) edge node[midway]{\textcolor{red}{10}}(13)
		  (12) edge (14)
		  (12) edge (15)
		  (16) edge node[midway]{\textcolor{red}{12}}(12)
		  (12) edge (17)
		  (18) edge node[midway]{\textcolor{red}{10}}(19)
	    (18) edge (20)
		(18) edge (21)
	    (22) edge node[midway]{\textcolor{red}{11}}(18)
	    (18) edge (23)
	    (0) edge node[midway]{\textcolor{red}{7}}(6)
	    (0) edge node[midway]{\textcolor{red}{8}}(12)
            (0) edge node[midway]{\textcolor{red}{9}}(18);
\path[loosely dotted] (4) edge (10)
		  (4) edge (16)
		(4) edge (22)
		  (1) edge (7)
		  (1) edge (13)
		  (1) edge (19)
		(2) edge (8)
		  (2) edge (14)
		(2) edge (20)
		(3) edge (9)
		  (3) edge (15)
		  (3) edge (21)
		(5) edge (11)
		(5) edge (17)
		(5) edge (23);
\end{tikzpicture}

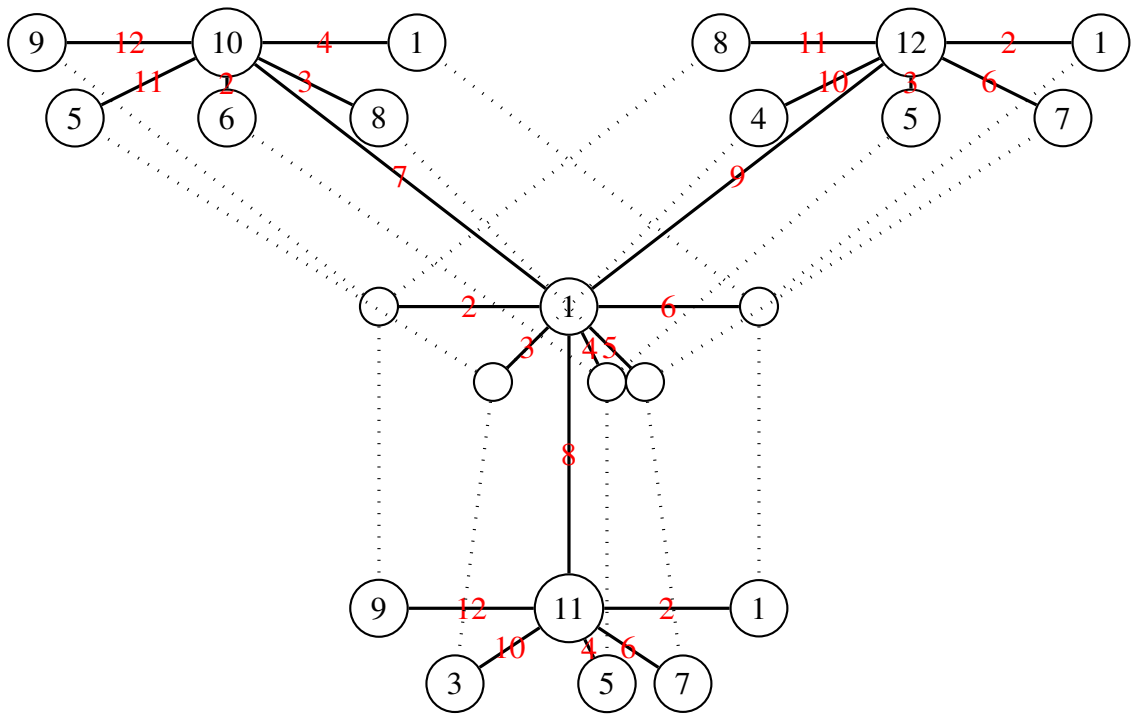
\begin{tikzpicture}[
every edge/.style = {draw=black,very thick},
 vrtx/.style args = {#1/#2}{%
      circle, draw, thick, fill=white,
      minimum size=5mm, label=#1:#2}, scale=0.5
                    ]
		\node   (0) [vrtx=left/] at (0, 8) {1};
		\node   (1) [vrtx=left/] at (-2, 6) {};
		\node   (2) [vrtx=left/] at (1, 6) {};
		\node   (3) [vrtx=left/] at (2, 6) {};
		\node   (4) [vrtx=left/] at (-5, 8) {};
		\node   (5) [vrtx=left/] at (5, 8) {};
		\node   (6) [vrtx=left/] at (-9, 15) {10};
		\node   (7) [vrtx=left/] at (-13, 13) {5};
		\node   (8) [vrtx=left/] at (-9, 13) {6};
		\node   (9) [vrtx=left/] at (-5, 13) {8};
		\node   (10) [vrtx=left/] at (-14, 15) {9};
		\node   (11) [vrtx=left/] at (-4, 15) {1};
		\node   (12) [vrtx=left/] at (0, 0) {11};
		\node   (13) [vrtx=left/] at (-3, -2) {3};
		\node   (14) [vrtx=left/] at (1, -2) {5};
		\node   (15) [vrtx=left/] at (3, -2) {7};
		\node   (16) [vrtx=left/] at (-5, 0) {9};
		\node   (17) [vrtx=left/] at (5, 0) {1};
		\node   (18) [vrtx=left/] at (9, 15) {12};
		\node   (19) [vrtx=left/] at (5, 13) {4};
		\node   (20) [vrtx=left/] at (9, 13) {5};
		\node   (21) [vrtx=left/] at (13, 13) {7};
		\node   (22) [vrtx=left/] at (4, 15) {8};
		\node   (23) [vrtx=left/] at (14, 15) {1};

        \path
        
		  (0) edge node[midway]{\textcolor{red}{3}}(1)
		  (0) edge node[midway]{\textcolor{red}{4}}(2)
		  (0) edge node[midway]{\textcolor{red}{5}}(3)
		(4) edge node[midway]{\textcolor{red}{2}}(0)
		  (0) edge node[midway]{\textcolor{red}{6}}(5)
		  (6) edge node[midway]{\textcolor{red}{11}}(7)
		  (6) edge node[midway]{\textcolor{red}{2}}(8)
		  (6) edge node[midway]{\textcolor{red}{3}}(9)
		(10) edge node[midway]{\textcolor{red}{12}}(6)
		  (6) edge node[midway]{\textcolor{red}{4}}(11)
		  (12) edge node[midway]{\textcolor{red}{10}}(13)
		  (12) edge node[midway]{\textcolor{red}{4}}(14)
		  (12) edge node[midway]{\textcolor{red}{6}}(15)
		  (16) edge node[midway]{\textcolor{red}{12}}(12)
		  (12) edge node[midway]{\textcolor{red}{2}}(17)
		  (18) edge node[midway]{\textcolor{red}{10}}(19)
	    (18) edge node[midway]{\textcolor{red}{3}}(20)
		(18) edge node[midway]{\textcolor{red}{6}}(21)
	    (22) edge node[midway]{\textcolor{red}{11}}(18)
	    (18) edge node[midway]{\textcolor{red}{2}}(23)
	    (0) edge node[midway]{\textcolor{red}{7}}(6)
	    (0) edge node[midway]{\textcolor{red}{8}}(12)
            (0) edge node[midway]{\textcolor{red}{9}}(18);
\path[loosely dotted] (4) edge (10)
		  (4) edge (16)
		(4) edge (22)
		  (1) edge (7)
		  (1) edge (13)
		  (1) edge (19)
		(2) edge (8)
		  (2) edge (14)
		(2) edge (20)
		(3) edge (9)
		  (3) edge (15)
		  (3) edge (21)
		(5) edge (11)
		(5) edge (17)
		(5) edge (23);
\end{tikzpicture}

In step 5, we focus on the skeleton corresponding to each edge of $(S_n)_{v_0}$. Here, we've highlighted that skeleton and colored it as instructed. Repeat this process for each edge in $(S_n)_{v_0}$. Clearly, sometimes we run into a problem where we use 10,11, or 12 when there already is a 10, 11, or 12 in the star $(S_n)_{v_j}$ for $j\neq0$.

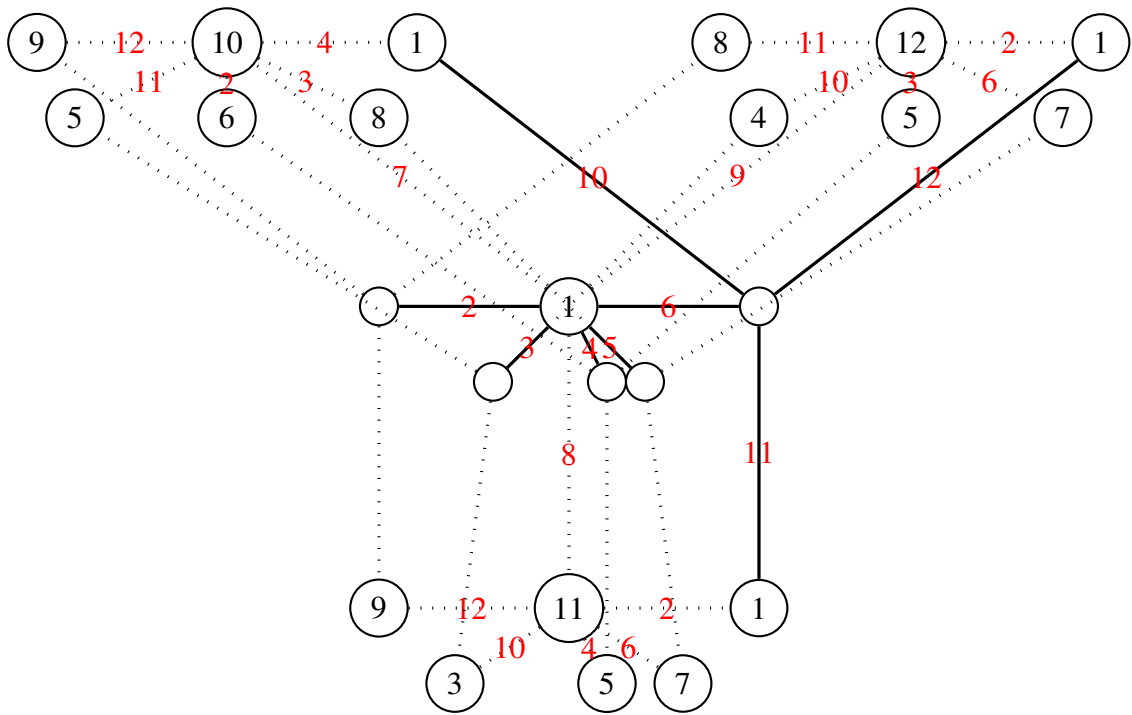
\begin{tikzpicture}[
every edge/.style = {draw=black,very thick},
 vrtx/.style args = {#1/#2}{%
      circle, draw, thick, fill=white,
      minimum size=5mm, label=#1:#2}, scale=0.5
                    ]
		\node   (0) [vrtx=left/] at (0, 8) {1};
		\node   (1) [vrtx=left/] at (-2, 6) {};
		\node   (2) [vrtx=left/] at (1, 6) {};
		\node   (3) [vrtx=left/] at (2, 6) {};
		\node   (4) [vrtx=left/] at (-5, 8) {};
		\node   (5) [vrtx=left/] at (5, 8) {};
		\node   (6) [vrtx=left/] at (-9, 15) {10};
		\node   (7) [vrtx=left/] at (-13, 13) {5};
		\node   (8) [vrtx=left/] at (-9, 13) {6};
		\node   (9) [vrtx=left/] at (-5, 13) {8};
		\node   (10) [vrtx=left/] at (-14, 15) {9};
		\node   (11) [vrtx=left/] at (-4, 15) {1};
		\node   (12) [vrtx=left/] at (0, 0) {11};
		\node   (13) [vrtx=left/] at (-3, -2) {3};
		\node   (14) [vrtx=left/] at (1, -2) {5};
		\node   (15) [vrtx=left/] at (3, -2) {7};
		\node   (16) [vrtx=left/] at (-5, 0) {9};
		\node   (17) [vrtx=left/] at (5, 0) {1};
		\node   (18) [vrtx=left/] at (9, 15) {12};
		\node   (19) [vrtx=left/] at (5, 13) {4};
		\node   (20) [vrtx=left/] at (9, 13) {5};
		\node   (21) [vrtx=left/] at (13, 13) {7};
		\node   (22) [vrtx=left/] at (4, 15) {8};
		\node   (23) [vrtx=left/] at (14, 15) {1};

        \path[loosely dotted]

		  (6) edge node[midway]{\textcolor{red}{11}}(7)
		  (6) edge node[midway]{\textcolor{red}{2}}(8)
		  (6) edge node[midway]{\textcolor{red}{3}}(9)
		(10) edge node[midway]{\textcolor{red}{12}}(6)
		  (6) edge node[midway]{\textcolor{red}{4}}(11)
		  (12) edge node[midway]{\textcolor{red}{10}}(13)
		  (12) edge node[midway]{\textcolor{red}{4}}(14)
		  (12) edge node[midway]{\textcolor{red}{6}}(15)
		  (16) edge node[midway]{\textcolor{red}{12}}(12)
		  (12) edge node[midway]{\textcolor{red}{2}}(17)
		  (18) edge node[midway]{\textcolor{red}{10}}(19)
	    (18) edge node[midway]{\textcolor{red}{3}}(20)
		(18) edge node[midway]{\textcolor{red}{6}}(21)
	    (22) edge node[midway]{\textcolor{red}{11}}(18)
	    (18) edge node[midway]{\textcolor{red}{2}}(23)
	    (0) edge node[midway]{\textcolor{red}{7}}(6)
	    (0) edge node[midway]{\textcolor{red}{8}}(12)
            (0) edge node[midway]{\textcolor{red}{9}}(18)
(4) edge (10)
		  (4) edge (16)
		(4) edge (22)
		  (1) edge (7)
		  (1) edge (13)
		  (1) edge (19)
		(2) edge (8)
		  (2) edge (14)
		(2) edge (20)
		(3) edge (9)
		  (3) edge (15)
		  (3) edge (21);
\path  (0) edge node[midway]{\textcolor{red}{3}}(1)
		  (0) edge node[midway]{\textcolor{red}{4}}(2)
		  (0) edge node[midway]{\textcolor{red}{5}}(3)
		(4) edge node[midway]{\textcolor{red}{2}}(0)
		  (0) edge node[midway]{\textcolor{red}{6}}(5)
          (5) edge node[midway]{\textcolor{red}{10}}(11)
		(5) edge node[midway]{\textcolor{red}{11}}(17)
		(5) edge node[midway]{\textcolor{red}{12}}(23);
\end{tikzpicture}

Shown below is an example of the steps taken in the algorithm for case 2 in Theorem~\ref{thm:line_tot} to achieve a maximal $b$-coloring. Here, we use $n=5$ and $m=4$, satisfying $n<2(m-1)$.

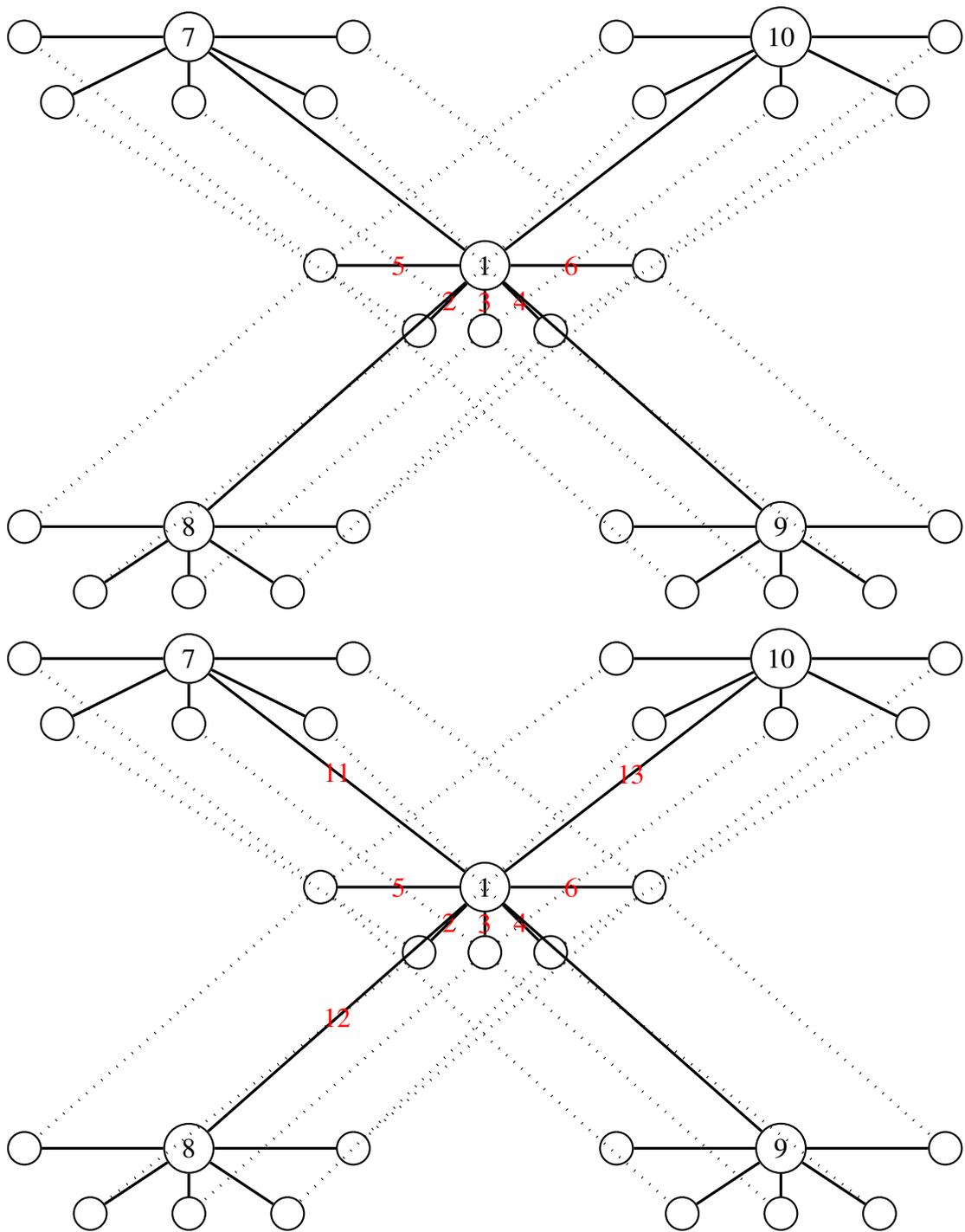
\begin{tikzpicture}[
every edge/.style = {draw=black,very thick},
 vrtx/.style args = {#1/#2}{%
      circle, draw, thick, fill=white,
      minimum size=5mm, label=#1:#2}, scale=0.5
                    ]
		\node   (0) [vrtx=left/] at (0, 8) {1};
		\node   (1) [vrtx=left/] at (-2, 6) {};
		\node   (2) [vrtx=left/] at (0, 6) {};
		\node   (3) [vrtx=left/] at (2, 6) {};
		\node   (4) [vrtx=left/] at (-5, 8) {};
		\node   (5) [vrtx=left/] at (5, 8) {};
		\node   (6) [vrtx=left/] at (-9, 15) {7};
		\node   (7) [vrtx=left/] at (-13, 13) {};
		\node   (8) [vrtx=left/] at (-9, 13) {};
		\node   (9) [vrtx=left/] at (-5, 13) {};
		\node   (10) [vrtx=left/] at (-14, 15) {};
		\node   (11) [vrtx=left/] at (-4, 15) {};
		\node   (12) [vrtx=left/] at (-9, 0) {8};
		\node   (13) [vrtx=left/] at (-12, -2) {};
		\node   (14) [vrtx=left/] at (-9, -2) {};
		\node   (15) [vrtx=left/] at (-6, -2) {};
		\node   (16) [vrtx=left/] at (-14, 0) {};
		\node   (17) [vrtx=left/] at (-4, 0) {};
		\node   (18) [vrtx=left/] at (9, 15) {10};
		\node   (19) [vrtx=left/] at (5, 13) {};
		\node   (20) [vrtx=left/] at (9, 13) {};
		\node   (21) [vrtx=left/] at (13, 13) {};
		\node   (22) [vrtx=left/] at (4, 15) {};
		\node   (23) [vrtx=left/] at (14, 15) {};
            \node   (24) [vrtx=left/] at (9, 0) {9};
		\node   (25) [vrtx=left/] at (12, -2) {};
		\node   (26) [vrtx=left/] at (9, -2) {};
		\node   (27) [vrtx=left/] at (6, -2) {};
		\node   (28) [vrtx=left/] at (14, 0) {};
		\node   (29) [vrtx=left/] at (4, 0) {};

        \path
        
		  (0) edge node[midway]{\textcolor{red}{2}}(1)
		  (0) edge node[midway]{\textcolor{red}{3}}(2)
		  (0) edge node[midway]{\textcolor{red}{4}}(3)
		(4) edge node[midway]{\textcolor{red}{5}}(0)
		  (0) edge node[midway]{\textcolor{red}{6}}(5)
		  (6) edge (7)
		  (6) edge (8)
		  (6) edge (9)
		(10) edge (6)
		  (6) edge (11)
		  (12) edge (13)
		  (12) edge (14)
		  (12) edge (15)
		  (16) edge (12)
		  (12) edge (17)
		  (18) edge (19)
	    (18) edge (20)
		(18) edge (21)
	    (22) edge (18)
	    (18) edge (23)
	    (0) edge node[midway]{\textcolor{red}{}}(6)
	    (0) edge node[midway]{\textcolor{red}{}}(12)
            (0) edge node[midway]{\textcolor{red}{}}(18)
             (24) edge (25)
            (24) edge (26)
            (24) edge (27)
            (24) edge (28)
            (24) edge (29)
            (0) edge (24);
\path[loosely dotted] (4) edge (10)
		  (4) edge (16)
		(4) edge (22)
            (4) edge (27)
		  (1) edge (7)
		  (1) edge (13)
		  (1) edge (19)
		(2) edge (8)
		  (2) edge (14)
		(2) edge (20)
            (2) edge (26)
		(3) edge (9)
		  (3) edge (15)
		  (3) edge (21)
            (3) edge (25)
            (5) edge (28)
		(5) edge (11)
		(5) edge (17)
		(5) edge (23);
\end{tikzpicture}

\begin{tikzpicture}[
every edge/.style = {draw=black,very thick},
 vrtx/.style args = {#1/#2}{%
      circle, draw, thick, fill=white,
      minimum size=5mm, label=#1:#2}, scale=0.5
                    ]
		\node   (0) [vrtx=left/] at (0, 8) {1};
		\node   (1) [vrtx=left/] at (-2, 6) {};
		\node   (2) [vrtx=left/] at (0, 6) {};
		\node   (3) [vrtx=left/] at (2, 6) {};
		\node   (4) [vrtx=left/] at (-5, 8) {};
		\node   (5) [vrtx=left/] at (5, 8) {};
		\node   (6) [vrtx=left/] at (-9, 15) {7};
		\node   (7) [vrtx=left/] at (-13, 13) {};
		\node   (8) [vrtx=left/] at (-9, 13) {};
		\node   (9) [vrtx=left/] at (-5, 13) {};
		\node   (10) [vrtx=left/] at (-14, 15) {};
		\node   (11) [vrtx=left/] at (-4, 15) {};
		\node   (12) [vrtx=left/] at (-9, 0) {8};
		\node   (13) [vrtx=left/] at (-12, -2) {};
		\node   (14) [vrtx=left/] at (-9, -2) {};
		\node   (15) [vrtx=left/] at (-6, -2) {};
		\node   (16) [vrtx=left/] at (-14, 0) {};
		\node   (17) [vrtx=left/] at (-4, 0) {};
		\node   (18) [vrtx=left/] at (9, 15) {10};
		\node   (19) [vrtx=left/] at (5, 13) {};
		\node   (20) [vrtx=left/] at (9, 13) {};
		\node   (21) [vrtx=left/] at (13, 13) {};
		\node   (22) [vrtx=left/] at (4, 15) {};
		\node   (23) [vrtx=left/] at (14, 15) {};
            \node   (24) [vrtx=left/] at (9, 0) {9};
		\node   (25) [vrtx=left/] at (12, -2) {};
		\node   (26) [vrtx=left/] at (9, -2) {};
		\node   (27) [vrtx=left/] at (6, -2) {};
		\node   (28) [vrtx=left/] at (14, 0) {};
		\node   (29) [vrtx=left/] at (4, 0) {};

        \path
        
		  (0) edge node[midway]{\textcolor{red}{2}}(1)
		  (0) edge node[midway]{\textcolor{red}{3}}(2)
		  (0) edge node[midway]{\textcolor{red}{4}}(3)
		(4) edge node[midway]{\textcolor{red}{5}}(0)
		  (0) edge node[midway]{\textcolor{red}{6}}(5)
		  (6) edge (7)
		  (6) edge (8)
		  (6) edge (9)
		(10) edge (6)
		  (6) edge (11)
		  (12) edge (13)
		  (12) edge (14)
		  (12) edge (15)
		  (16) edge (12)
		  (12) edge (17)
		  (18) edge (19)
	    (18) edge (20)
		(18) edge (21)
	    (22) edge (18)
	    (18) edge (23)
	    (0) edge node[midway]{\textcolor{red}{11}}(6)
	    (0) edge node[midway]{\textcolor{red}{12}}(12)
            (0) edge node[midway]{\textcolor{red}{13}}(18)
             (24) edge (25)
            (24) edge (26)
            (24) edge (27)
            (24) edge (28)
            (24) edge (29)
            (0) edge (24);
\path[loosely dotted] (4) edge (10)
		  (4) edge (16)
		(4) edge (22)
            (4) edge (27)
		  (1) edge (7)
		  (1) edge (13)
		  (1) edge (19)
		(2) edge (8)
		  (2) edge (14)
		(2) edge (20)
            (2) edge (26)
		(3) edge (9)
		  (3) edge (15)
		  (3) edge (21)
            (3) edge (25)
            (5) edge (28)
		(5) edge (11)
		(5) edge (17)
		(5) edge (23);
\end{tikzpicture}

In step 4, we use the mapping $\phi$ to color the outer vertices of $(S_n)_{v_j}$ such that $j\neq0$.

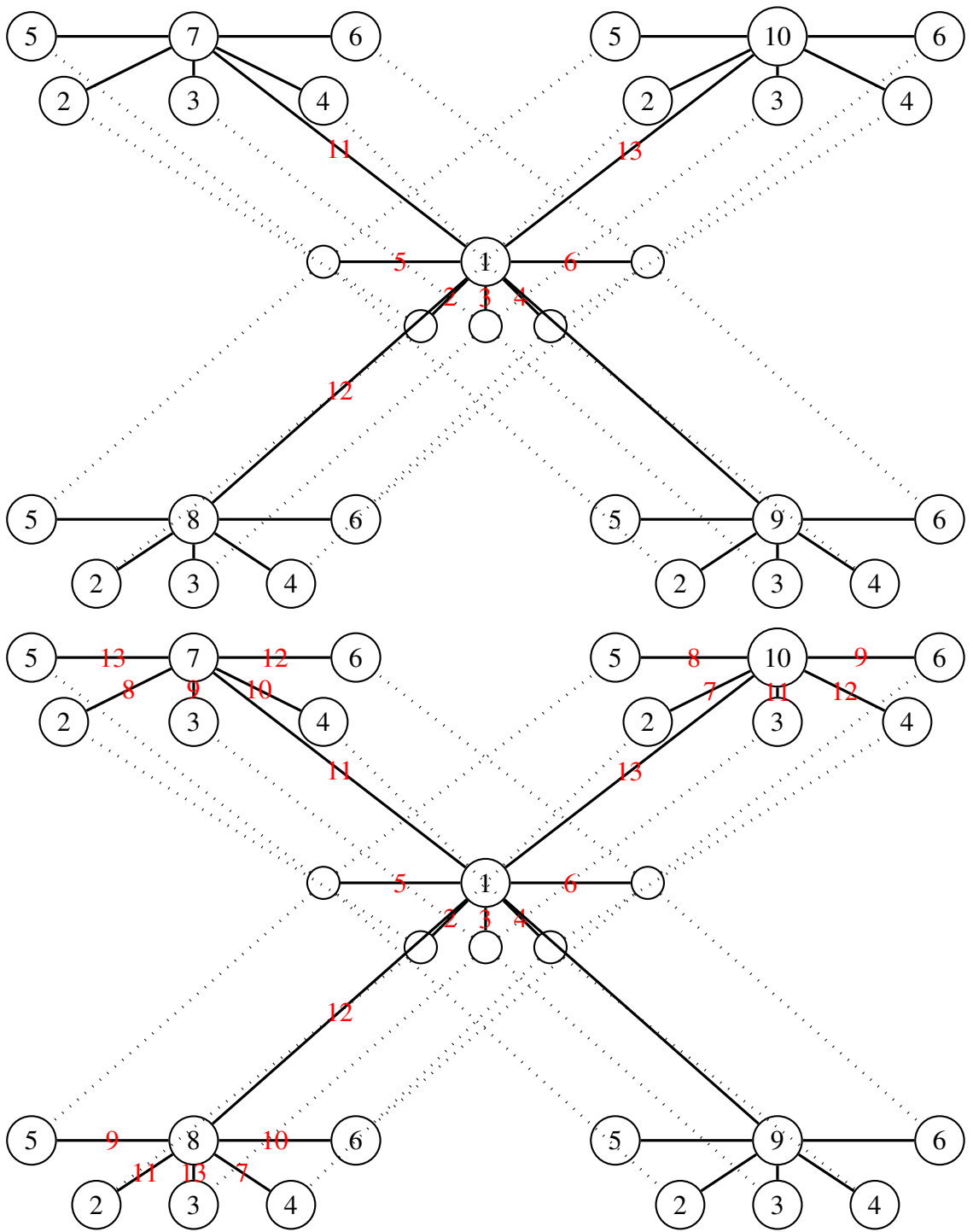
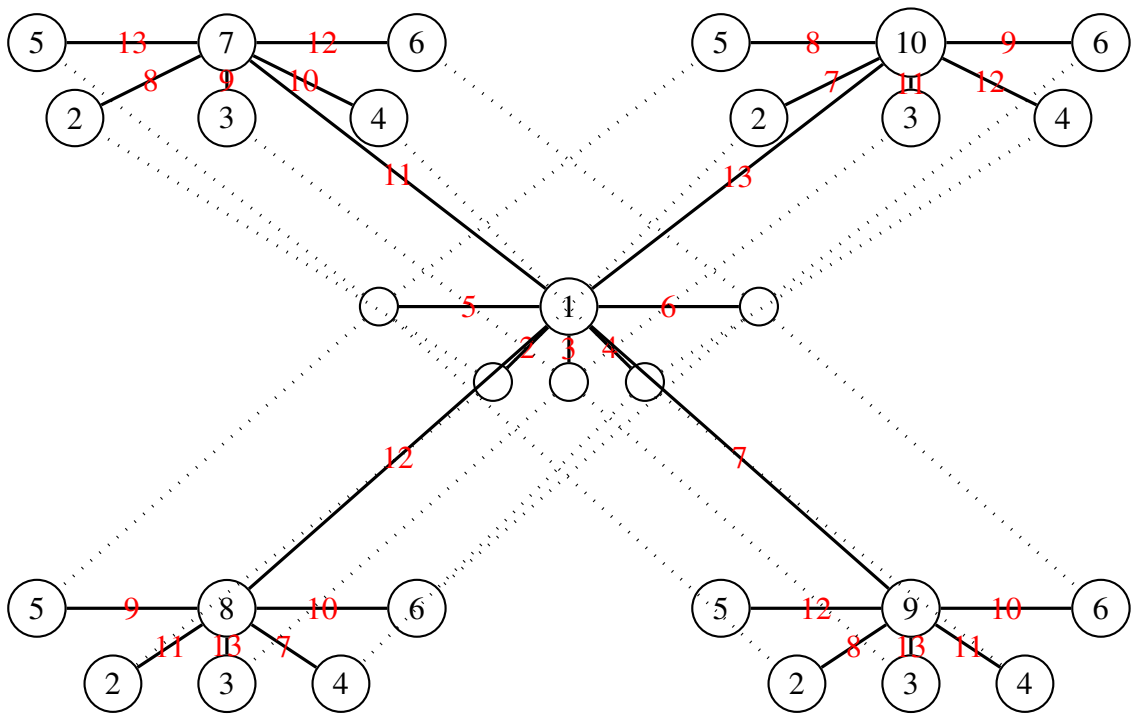
\begin{tikzpicture}[
every edge/.style = {draw=black,very thick},
 vrtx/.style args = {#1/#2}{%
      circle, draw, thick, fill=white,
      minimum size=5mm, label=#1:#2}, scale=0.5
                    ]
		\node   (0) [vrtx=left/] at (0, 8) {1};
		\node   (1) [vrtx=left/] at (-2, 6) {};
		\node   (2) [vrtx=left/] at (0, 6) {};
		\node   (3) [vrtx=left/] at (2, 6) {};
		\node   (4) [vrtx=left/] at (-5, 8) {};
		\node   (5) [vrtx=left/] at (5, 8) {};
		\node   (6) [vrtx=left/] at (-9, 15) {7};
		\node   (7) [vrtx=left/] at (-13, 13) {2};
		\node   (8) [vrtx=left/] at (-9, 13) {3};
		\node   (9) [vrtx=left/] at (-5, 13) {4};
		\node   (10) [vrtx=left/] at (-14, 15) {5};
		\node   (11) [vrtx=left/] at (-4, 15) {6};
		\node   (12) [vrtx=left/] at (-9, 0) {8};
		\node   (13) [vrtx=left/] at (-12, -2) {2};
		\node   (14) [vrtx=left/] at (-9, -2) {3};
		\node   (15) [vrtx=left/] at (-6, -2) {4};
		\node   (16) [vrtx=left/] at (-14, 0) {5};
		\node   (17) [vrtx=left/] at (-4, 0) {6};
		\node   (18) [vrtx=left/] at (9, 15) {10};
		\node   (19) [vrtx=left/] at (5, 13) {2};
		\node   (20) [vrtx=left/] at (9, 13) {3};
		\node   (21) [vrtx=left/] at (13, 13) {4};
		\node   (22) [vrtx=left/] at (4, 15) {5};
		\node   (23) [vrtx=left/] at (14, 15) {6};
            \node   (24) [vrtx=left/] at (9, 0) {9};
		\node   (25) [vrtx=left/] at (12, -2) {4};
		\node   (26) [vrtx=left/] at (9, -2) {3};
		\node   (27) [vrtx=left/] at (6, -2) {2};
		\node   (28) [vrtx=left/] at (14, 0) {6};
		\node   (29) [vrtx=left/] at (4, 0) {5};

        \path
        
		  (0) edge node[midway]{\textcolor{red}{2}}(1)
		  (0) edge node[midway]{\textcolor{red}{3}}(2)
		  (0) edge node[midway]{\textcolor{red}{4}}(3)
		(4) edge node[midway]{\textcolor{red}{5}}(0)
		  (0) edge node[midway]{\textcolor{red}{6}}(5)
		  (6) edge (7)
		  (6) edge (8)
		  (6) edge (9)
		(10) edge (6)
		  (6) edge (11)
		  (12) edge (13)
		  (12) edge (14)
		  (12) edge (15)
		  (16) edge (12)
		  (12) edge (17)
		  (18) edge (19)
	    (18) edge (20)
		(18) edge (21)
	    (22) edge (18)
	    (18) edge (23)
	    (0) edge node[midway]{\textcolor{red}{11}}(6)
	    (0) edge node[midway]{\textcolor{red}{12}}(12)
            (0) edge node[midway]{\textcolor{red}{13}}(18)
             (24) edge (25)
            (24) edge (26)
            (24) edge (27)
            (24) edge (28)
            (24) edge (29)
            (0) edge (24);
\path[loosely dotted] (4) edge (10)
		  (4) edge (16)
		(4) edge (22)
            (4) edge (27)
		  (1) edge (7)
		  (1) edge (13)
		  (1) edge (19)
		(2) edge (8)
		  (2) edge (14)
		(2) edge (20)
            (2) edge (26)
		(3) edge (9)
		  (3) edge (15)
		  (3) edge (21)
            (3) edge (25)
            (5) edge (28)
		(5) edge (11)
		(5) edge (17)
		(5) edge (23);
\end{tikzpicture}

\begin{tikzpicture}[
every edge/.style = {draw=black,very thick},
 vrtx/.style args = {#1/#2}{%
      circle, draw, thick, fill=white,
      minimum size=5mm, label=#1:#2}, scale=0.5
                    ]
		\node   (0) [vrtx=left/] at (0, 8) {1};
		\node   (1) [vrtx=left/] at (-2, 6) {};
		\node   (2) [vrtx=left/] at (0, 6) {};
		\node   (3) [vrtx=left/] at (2, 6) {};
		\node   (4) [vrtx=left/] at (-5, 8) {};
		\node   (5) [vrtx=left/] at (5, 8) {};
		\node   (6) [vrtx=left/] at (-9, 15) {7};
		\node   (7) [vrtx=left/] at (-13, 13) {2};
		\node   (8) [vrtx=left/] at (-9, 13) {3};
		\node   (9) [vrtx=left/] at (-5, 13) {4};
		\node   (10) [vrtx=left/] at (-14, 15) {5};
		\node   (11) [vrtx=left/] at (-4, 15) {6};
		\node   (12) [vrtx=left/] at (-9, 0) {8};
		\node   (13) [vrtx=left/] at (-12, -2) {2};
		\node   (14) [vrtx=left/] at (-9, -2) {3};
		\node   (15) [vrtx=left/] at (-6, -2) {4};
		\node   (16) [vrtx=left/] at (-14, 0) {5};
		\node   (17) [vrtx=left/] at (-4, 0) {6};
		\node   (18) [vrtx=left/] at (9, 15) {10};
		\node   (19) [vrtx=left/] at (5, 13) {2};
		\node   (20) [vrtx=left/] at (9, 13) {3};
		\node   (21) [vrtx=left/] at (13, 13) {4};
		\node   (22) [vrtx=left/] at (4, 15) {5};
		\node   (23) [vrtx=left/] at (14, 15) {6};
            \node   (24) [vrtx=left/] at (9, 0) {9};
		\node   (25) [vrtx=left/] at (12, -2) {4};
		\node   (26) [vrtx=left/] at (9, -2) {3};
		\node   (27) [vrtx=left/] at (6, -2) {2};
		\node   (28) [vrtx=left/] at (14, 0) {6};
		\node   (29) [vrtx=left/] at (4, 0) {5};

        \path
        
		  (0) edge node[midway]{\textcolor{red}{2}}(1)
		  (0) edge node[midway]{\textcolor{red}{3}}(2)
		  (0) edge node[midway]{\textcolor{red}{4}}(3)
		(4) edge node[midway]{\textcolor{red}{5}}(0)
		  (0) edge node[midway]{\textcolor{red}{6}}(5)
		  (6) edge node[midway]{\textcolor{red}{8}}(7)
		  (6) edge node[midway]{\textcolor{red}{9}}(8)
		  (6) edge node[midway]{\textcolor{red}{10}}(9)
		(10) edge node[midway]{\textcolor{red}{13}}(6)
		  (6) edge node[midway]{\textcolor{red}{12}}(11)
		  (12) edge node[midway]{\textcolor{red}{11}}(13)
		  (12) edge node[midway]{\textcolor{red}{13}}(14)
		  (12) edge node[midway]{\textcolor{red}{7}}(15)
		  (16) edge node[midway]{\textcolor{red}{9}}(12)
		  (12) edge node[midway]{\textcolor{red}{10}}(17)
		  (18) edge node[midway]{\textcolor{red}{7}}(19)
	    (18) edge node[midway]{\textcolor{red}{11}}(20)
		(18) edge node[midway]{\textcolor{red}{12}}(21)
	    (22) edge node[midway]{\textcolor{red}{8}}(18)
	    (18) edge node[midway]{\textcolor{red}{9}}(23)
	    (0) edge node[midway]{\textcolor{red}{11}}(6)
	    (0) edge node[midway]{\textcolor{red}{12}}(12)
            (0) edge node[midway]{\textcolor{red}{13}}(18)
             (24) edge (25)
            (24) edge (26)
            (24) edge (27)
            (24) edge (28)
            (24) edge (29)
            (0) edge (24);
\path[loosely dotted] (4) edge (10)
		  (4) edge (16)
		(4) edge (22)
            (4) edge (27)
		  (1) edge (7)
		  (1) edge (13)
		  (1) edge (19)
		(2) edge (8)
		  (2) edge (14)
		(2) edge (20)
            (2) edge (26)
		(3) edge (9)
		  (3) edge (15)
		  (3) edge (21)
            (3) edge (25)
            (5) edge (28)
		(5) edge (11)
		(5) edge (17)
		(5) edge (23);
\end{tikzpicture}

\begin{tikzpicture}[
every edge/.style = {draw=black,very thick},
 vrtx/.style args = {#1/#2}{%
      circle, draw, thick, fill=white,
      minimum size=5mm, label=#1:#2}, scale=0.5
                    ]
		\node   (0) [vrtx=left/] at (0, 8) {1};
		\node   (1) [vrtx=left/] at (-2, 6) {};
		\node   (2) [vrtx=left/] at (0, 6) {};
		\node   (3) [vrtx=left/] at (2, 6) {};
		\node   (4) [vrtx=left/] at (-5, 8) {};
		\node   (5) [vrtx=left/] at (5, 8) {};
		\node   (6) [vrtx=left/] at (-9, 15) {7};
		\node   (7) [vrtx=left/] at (-13, 13) {2};
		\node   (8) [vrtx=left/] at (-9, 13) {3};
		\node   (9) [vrtx=left/] at (-5, 13) {4};
		\node   (10) [vrtx=left/] at (-14, 15) {5};
		\node   (11) [vrtx=left/] at (-4, 15) {6};
		\node   (12) [vrtx=left/] at (-9, 0) {8};
		\node   (13) [vrtx=left/] at (-12, -2) {2};
		\node   (14) [vrtx=left/] at (-9, -2) {3};
		\node   (15) [vrtx=left/] at (-6, -2) {4};
		\node   (16) [vrtx=left/] at (-14, 0) {5};
		\node   (17) [vrtx=left/] at (-4, 0) {6};
		\node   (18) [vrtx=left/] at (9, 15) {10};
		\node   (19) [vrtx=left/] at (5, 13) {2};
		\node   (20) [vrtx=left/] at (9, 13) {3};
		\node   (21) [vrtx=left/] at (13, 13) {4};
		\node   (22) [vrtx=left/] at (4, 15) {5};
		\node   (23) [vrtx=left/] at (14, 15) {6};
            \node   (24) [vrtx=left/] at (9, 0) {9};
		\node   (25) [vrtx=left/] at (12, -2) {4};
		\node   (26) [vrtx=left/] at (9, -2) {3};
		\node   (27) [vrtx=left/] at (6, -2) {2};
		\node   (28) [vrtx=left/] at (14, 0) {6};
		\node   (29) [vrtx=left/] at (4, 0) {5};

        \path
        
		  (0) edge node[midway]{\textcolor{red}{2}}(1)
		  (0) edge node[midway]{\textcolor{red}{3}}(2)
		  (0) edge node[midway]{\textcolor{red}{4}}(3)
		(4) edge node[midway]{\textcolor{red}{5}}(0)
		  (0) edge node[midway]{\textcolor{red}{6}}(5)
		  (6) edge node[midway]{\textcolor{red}{8}}(7)
		  (6) edge node[midway]{\textcolor{red}{9}}(8)
		  (6) edge node[midway]{\textcolor{red}{10}}(9)
		(10) edge node[midway]{\textcolor{red}{13}}(6)
		  (6) edge node[midway]{\textcolor{red}{12}}(11)
		  (12) edge node[midway]{\textcolor{red}{11}}(13)
		  (12) edge node[midway]{\textcolor{red}{13}}(14)
		  (12) edge node[midway]{\textcolor{red}{7}}(15)
		  (16) edge node[midway]{\textcolor{red}{9}}(12)
		  (12) edge node[midway]{\textcolor{red}{10}}(17)
		  (18) edge node[midway]{\textcolor{red}{7}}(19)
	    (18) edge node[midway]{\textcolor{red}{11}}(20)
		(18) edge node[midway]{\textcolor{red}{12}}(21)
	    (22) edge node[midway]{\textcolor{red}{8}}(18)
	    (18) edge node[midway]{\textcolor{red}{9}}(23)
	    (0) edge node[midway]{\textcolor{red}{11}}(6)
	    (0) edge node[midway]{\textcolor{red}{12}}(12)
            (0) edge node[midway]{\textcolor{red}{13}}(18)
             (24) edge node[midway]{\textcolor{red}{11}}(25)
            (24) edge node[midway]{\textcolor{red}{13}}(26)
            (24) edge node[midway]{\textcolor{red}{8}}(27)
            (24) edge node[midway]{\textcolor{red}{10}}(28)
            (24) edge node[midway]{\textcolor{red}{12}}(29)
            (0) edge node[midway]{\textcolor{red}{7}}(24);
\path[loosely dotted] (4) edge (10)
		  (4) edge (16)
		(4) edge (22)
            (4) edge (27)
		  (1) edge (7)
		  (1) edge (13)
		  (1) edge (19)
		(2) edge (8)
		  (2) edge (14)
		(2) edge (20)
            (2) edge (26)
		(3) edge (9)
		  (3) edge (15)
		  (3) edge (21)
            (3) edge (25)
            (5) edge (28)
		(5) edge (11)
		(5) edge (17)
		(5) edge (23);
\end{tikzpicture}

In step 8, we color $m$ edges at a time based off $c((w_0)_{v_j})$, $j\neq0$. Repeat this process for each edge in $(S_n)_{v_0}$. The edges have no relation to each other, so we can use the same colors each iteration.

\begin{tikzpicture}[
every edge/.style = {draw=black,very thick},
 vrtx/.style args = {#1/#2}{%
      circle, draw, thick, fill=white,
      minimum size=5mm, label=#1:#2}, scale=0.5
                    ]
		\node   (0) [vrtx=left/] at (0, 8) {1};
		\node   (1) [vrtx=left/] at (-2, 6) {};
		\node   (2) [vrtx=left/] at (0, 6) {};
		\node   (3) [vrtx=left/] at (2, 6) {};
		\node   (4) [vrtx=left/] at (-5, 8) {};
		\node   (5) [vrtx=left/] at (5, 8) {};
		\node   (6) [vrtx=left/] at (-9, 15) {7};
		\node   (7) [vrtx=left/] at (-13, 13) {2};
		\node   (8) [vrtx=left/] at (-9, 13) {3};
		\node   (9) [vrtx=left/] at (-5, 13) {4};
		\node   (10) [vrtx=left/] at (-14, 15) {5};
		\node   (11) [vrtx=left/] at (-4, 15) {6};
		\node   (12) [vrtx=left/] at (-9, 0) {8};
		\node   (13) [vrtx=left/] at (-12, -2) {2};
		\node   (14) [vrtx=left/] at (-9, -2) {3};
		\node   (15) [vrtx=left/] at (-6, -2) {4};
		\node   (16) [vrtx=left/] at (-14, 0) {5};
		\node   (17) [vrtx=left/] at (-4, 0) {6};
		\node   (18) [vrtx=left/] at (9, 15) {10};
		\node   (19) [vrtx=left/] at (5, 13) {2};
		\node   (20) [vrtx=left/] at (9, 13) {3};
		\node   (21) [vrtx=left/] at (13, 13) {4};
		\node   (22) [vrtx=left/] at (4, 15) {5};
		\node   (23) [vrtx=left/] at (14, 15) {6};
            \node   (24) [vrtx=left/] at (9, 0) {9};
		\node   (25) [vrtx=left/] at (12, -2) {4};
		\node   (26) [vrtx=left/] at (9, -2) {3};
		\node   (27) [vrtx=left/] at (6, -2) {2};
		\node   (28) [vrtx=left/] at (14, 0) {6};
		\node   (29) [vrtx=left/] at (4, 0) {5};

        \path
        
		  (0) edge node[midway]{\textcolor{red}{2}}(1)
		  (0) edge node[midway]{\textcolor{red}{3}}(2)
		  (0) edge node[midway]{\textcolor{red}{4}}(3)
		(4) edge node[midway]{\textcolor{red}{5}}(0)
		  (0) edge node[midway]{\textcolor{red}{6}}(5)
            (5) edge node[midway]{\textcolor{red}{9}}(28)
		(5) edge node[midway]{\textcolor{red}{7}}(11)
		(5) edge node[midway]{\textcolor{red}{8}}(17)
		(5) edge node[midway]{\textcolor{red}{10}}(23);
		  \path[loosely dotted]
            (6) edge node[midway]{\textcolor{red}{8}}(7)
		  (6) edge node[midway]{\textcolor{red}{9}}(8)
		  (6) edge node[midway]{\textcolor{red}{10}}(9)
		(10) edge node[midway]{\textcolor{red}{13}}(6)
		  (6) edge node[midway]{\textcolor{red}{12}}(11)
		  (12) edge node[midway]{\textcolor{red}{11}}(13)
		  (12) edge node[midway]{\textcolor{red}{13}}(14)
		  (12) edge node[midway]{\textcolor{red}{7}}(15)
		  (16) edge node[midway]{\textcolor{red}{9}}(12)
		  (12) edge node[midway]{\textcolor{red}{10}}(17)
		  (18) edge node[midway]{\textcolor{red}{7}}(19)
	    (18) edge node[midway]{\textcolor{red}{11}}(20)
		(18) edge node[midway]{\textcolor{red}{12}}(21)
	    (22) edge node[midway]{\textcolor{red}{8}}(18)
	    (18) edge node[midway]{\textcolor{red}{9}}(23)
	    (0) edge node[midway]{\textcolor{red}{11}}(6)
	    (0) edge node[midway]{\textcolor{red}{12}}(12)
            (0) edge node[midway]{\textcolor{red}{13}}(18)
             (24) edge node[midway]{\textcolor{red}{11}}(25)
            (24) edge node[midway]{\textcolor{red}{13}}(26)
            (24) edge node[midway]{\textcolor{red}{8}}(27)
            (24) edge node[midway]{\textcolor{red}{10}}(28)
            (24) edge node[midway]{\textcolor{red}{12}}(29)
            (0) edge node[midway]{\textcolor{red}{7}}(24)
            (4) edge (10)
		  (4) edge (16)
		(4) edge (22)
            (4) edge (27)
		  (1) edge (7)
		  (1) edge (13)
		  (1) edge (19)
		(2) edge (8)
		  (2) edge (14)
		(2) edge (20)
            (2) edge (26)
		(3) edge (9)
		  (3) edge (15)
		  (3) edge (21)
            (3) edge (25);
\end{tikzpicture}
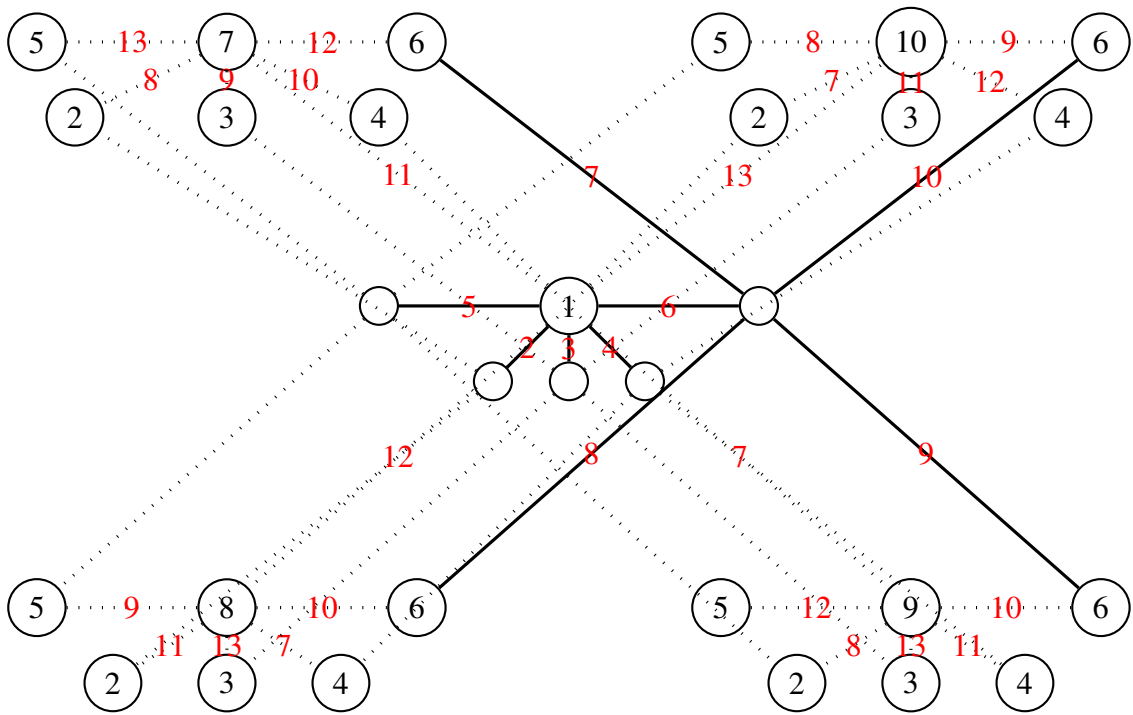
\end{document}